\documentclass[11pt,reqno]{amsart}
\usepackage[T1]{fontenc}
\usepackage[utf8]{inputenc}
\usepackage{lmodern}
\usepackage[a4paper,margin=29mm]{geometry}
\usepackage{amsmath,amssymb,amsthm,mathtools}
\usepackage{microtype}
\usepackage{enumitem}
\usepackage{xcolor}
\definecolor{linkblue}{RGB}{30,65,110}
\usepackage[colorlinks=true,linkcolor=linkblue,citecolor=linkblue,
 urlcolor=linkblue,breaklinks=true]{hyperref}
\usepackage{url}
\numberwithin{equation}{section}
\newtheorem{theorem}{Theorem}[section]
\newtheorem{proposition}[theorem]{Proposition}
\newtheorem{lemma}[theorem]{Lemma}
\newtheorem{corollary}[theorem]{Corollary}
\theoremstyle{definition}

\theoremstyle{remark}

\allowdisplaybreaks[2]
\setlist[enumerate]{leftmargin=*,itemsep=3pt,topsep=5pt}
\setlist[itemize]{leftmargin=*,itemsep=3pt,topsep=5pt}
\title[Logarithmic concentration on the two-dimensional torus]
{Logarithmic concentration for the focusing cubic Schr\"odinger equation on the two-dimensional torus}
\author{Xi Chen}
\date{}
\address{Department of Mathematics and Computer Science, University of Basel, Spiegelgasse 1, 4051 Basel, Switzerland}
\subjclass[2020]{35Q55, 35B40, 35B44, 35B45}
\keywords{Focusing nonlinear Schr\"odinger equation, critical mass,
infinite-time concentration, Sobolev growth, terminal-value construction}
\begin{document}
\begin{abstract}
We construct a smooth solution of the focusing cubic nonlinear
Schr\"odinger equation on a square two-dimensional torus which is global
in both time directions and concentrates at one point as time tends to
positive infinity. Its mass is the mass of the Euclidean ground state
and its energy is zero. If \(\lambda(t)\) denotes its concentration
scale and \(L\) the period, then
\(L/\lambda(t)=2\log t+\frac{11}{2}\log\log t+O(1)\).
The solution approaches a phased, periodized ground state with an error
tending to zero in every fixed Sobolev space. In particular, its
\(H^s\) norm is asymptotic to an explicit constant times
\((\log t)^s\) for every \(s>0\). We also prove uniqueness, up to
constant phase and time translation, within the quantitative
single-profile asymptotic class used in the construction.
The mechanism is the exponentially small interaction of the profile
with its periodic images. Exact mass and energy constraints select
the concentrating branch, while a positive energy on the moving
constrained space controls the terminal-value problem at every fixed
regularity order.
\end{abstract}
\maketitle

\section{Introduction}\label{sec:introduction}

We consider the focusing cubic nonlinear Schr\"odinger equation
\begin{equation}\label{eq:nls}
 i\partial_tu+\Delta u+|u|^2u=0,
 \qquad (t,x)\in\mathbb R\times\mathbb T_L^2,\qquad
 \mathbb T_L^2=\mathbb R^2/(L\mathbb Z^2).
\end{equation}
For a regular solution the mass and energy are conserved:
\begin{equation}\label{eq:conservation}
 M(u)=\int_{\mathbb T_L^2}|u|^2,\qquad
 E(u)=\frac12\int_{\mathbb T_L^2}|\nabla u|^2
                  -\frac14\int_{\mathbb T_L^2}|u|^4 .
\end{equation}
All integrals use unnormalized Lebesgue measure. The Euclidean ground
state is the positive radial solution \(Q\) of
\[
 -\Delta Q+Q-Q^3=0\quad\hbox{on }\mathbb R^2.
\]
For a domain \(\Omega\), the subscript in
\(M_\Omega(v)=\int_\Omega |v|^2\) specifies where the mass is taken.
We write \(m=M_{\mathbb R^2}(Q)\). The sharp
Gagliardo--Nirenberg inequality identifies \(m\) as the critical
mass for concentration \cite{Weinstein83}; the uniqueness of the
positive ground state is due to Kwong \cite{Kwong89}.

The purpose of this paper is to construct a critical-mass solution which
concentrates only at infinite time. Its growing Sobolev norms come from
a single shrinking spatial profile. Thus growth already occurs in
\(H^1\), where the focusing energy is no longer coercive at mass \(m\).
The solution is smooth at every finite time, and one datum gives the
asymptotic formula simultaneously at all fixed Sobolev orders.

For \(\lambda>0\), let
\[
 \mathcal Q_\lambda(x)=\lambda^{-1}
       \sum_{j\in\mathbb Z^2}Q\!\left(\frac{x-Lj}{\lambda}\right).
\]
We use the torus Fourier convention
\[
 \widehat v(n)=L^{-2}\int_{\mathbb T_L^2}
                    v(x)e^{-2\pi in\cdot x/L}\,dx,\qquad
 \|v\|_{H^s}^2=L^2\sum_{n\in\mathbb Z^2}
       (1+|2\pi n/L|^2)^s|\widehat v(n)|^2 .
\]
The homogeneous norm is defined by omitting the \(1\).
The Euclidean Fourier transform is
\(\widehat f(\xi)=\int_{\mathbb R^2}f(x)e^{-ix\cdot\xi}\,dx\),
with Parseval factor \((2\pi)^{-2}\).

In particular,
\[
 \|f\|_{\dot H^s(\mathbb R^2)}^2
   =(2\pi)^{-2}\int_{\mathbb R^2}|\xi|^{2s}
                                      |\widehat f(\xi)|^2\,d\xi .
\]

\begin{theorem}\label{thm:main}
For every \(L>0\), there is a datum
\(u_0\in C^\infty(\mathbb T_L^2)\) whose solution of
\eqref{eq:nls} belongs to
\[
 C^\infty(\mathbb R\times\mathbb T_L^2)
 \cap\bigcap_{k\ge0}C(\mathbb R;H^k(\mathbb T_L^2))
\]
and has \(M(u)=m\), \(E(u)=0\).
The datum and the solution are invariant under the two coordinate
reflections and the exchange of the coordinates.
There are smooth functions \(\lambda(t)>0\) and \(\gamma(t)\),
defined for all sufficiently large positive times, such that
\begin{equation}\label{intro:scale}
 \frac{L}{\lambda(t)}
 =2\log t+\frac{11}{2}\log\log t+O(1)
                  \qquad(t\to+\infty)
\end{equation}
and, for every fixed integer \(k\ge0\),
\begin{equation}\label{intro:profile}
 \|u(t)-e^{i\gamma(t)}\mathcal Q_{\lambda(t)}\|_{H^k}
                         \longrightarrow0.
\end{equation}
Consequently, for every fixed real \(s>0\),
\begin{equation}\label{intro:growth}
 \|u(t)\|_{H^s(\mathbb T_L^2)}
 \sim \|Q\|_{\dot H^s(\mathbb R^2)}
               \left(\frac{2\log t}{L}\right)^s .
\end{equation}
\end{theorem}

The error in \eqref{intro:profile} tends to zero in the physical
Sobolev norm; it is stronger than convergence after rescaling in
\(H^1\). The constants in its proof may depend on \(k\).
No uniform assertion as \(k\to\infty\) is intended.
We also obtain the constant term in \eqref{intro:scale} and a
convergent phase correction; see
\eqref{con:scale-refined} and \eqref{con:phase-limit}.

The theorem concerns the standard autonomous equation on a compact
domain. Its mechanism differs from transfer between many weakly
interacting Fourier modes: the mass concentrates at one point, and
the gradient norm itself diverges. The distinction is relevant when
comparing it with the following results on long-time growth and
concentration.

The long-time behavior of nonlinear Schr\"odinger equations on compact spaces
has two rather different aspects.  The local theory, developed on tori by
Bourgain~\cite{Bourgain93} and on compact manifolds by Burq, G\'erard and
Tzvetkov~\cite{BGT04}, can be combined with conservation laws to give global
solutions for the defocusing cubic equation in dimension two.  It gives much less
information about the distribution of a solution among high frequencies.
Bourgain's work on upper bounds for Sobolev norms~\cite{Bourgain96} and his
discussion of long-time growth~\cite{Bourgain00} helped formulate this question.
For the cubic defocusing equation on \(\mathbb T^2\), Colliander, Keel,
Staffilani, Takaoka and Tao~\cite{CKSTT10} constructed smooth solutions whose
\(H^s\) norm, for \(s>1\), grows from an arbitrarily small value to an
arbitrarily large prescribed value.  Guardia and Kaloshin~\cite{GK15} obtained
polynomial upper bounds on the time needed for a prescribed relative
amplification.  These are finite-amplification results: the initial datum
depends on the desired growth.

There are settings in which a single orbit is known to have unbounded Sobolev
norms.  Hani~\cite{Hani14} proved this for a family of periodic cubic
Hamiltonian equations obtained by retaining interactions with bounded
nonresonance; the family includes the resonant equation, but its
infinite-growth theorem does not include the full cubic nonlinearity.
On \(\mathbb R\times\mathbb T^d\), Hani, Pausader, Tzvetkov and
Visciglia~\cite{HPTV15} used modified scattering to realize unbounded resonant
orbits as solutions of cubic NLS for \(d\ge2\), with unbounded norms at
sufficiently high Sobolev regularity.
Another source of examples is the cubic Szeg\H{o} equation, introduced by
G\'erard and Grellier~\cite{GG10}.  Their nonlinear Fourier
transform~\cite{GG17} yields smooth solutions on the circle with unbounded
\(H^s\) norms for every \(s>1/2\) along sequences of times, whereas
Pocovnicu~\cite{Pocovnicu11} obtained explicit rational examples on the real
line.  The connection with the defocusing cubic half-wave equation is
quantitative on long finite intervals for small periodic data supported on
nonnegative Fourier modes~\cite{GG12}.  For the focusing half-wave equation on the
line, G\'erard, Lenzmann, Pocovnicu and Rapha\"el~\cite{GLPR18} constructed a
two-soliton solution with a large transient increase of its \(H^1\) norm
followed by saturation. The infinite cascades for the half wave Schr\"odinger equation have been eastablished in \cite{ChenPlane,XuWaveguide}. More recently, the author has proved the infinite cascades for the cubic half-wave equation on the line and on the torus \cite{Chen26.1,Chen26.2}.

The focusing equation also permits growth through concentration in physical
space.  On \(\mathbb R^2\), the sharp interpolation inequality of
Weinstein~\cite{Weinstein83}, together with the uniqueness theorem of
Kwong~\cite{Kwong89}, identifies the ground state \(Q\) and the mass
\(\|Q\|_2^2\) that governs the critical problem.  The standard Cauchy theory
and conservation laws are reviewed in~\cite{Cazenave03}.
Below this mass, Dodson~\cite{Dodson15} proved global existence and scattering
for all \(L^2\) data.  At the threshold, Merle~\cite{Merle93} classified
finite-time \(H^1\) blow-up solutions as pseudoconformal transforms of the ground
state, up to the symmetries of the Euclidean equation.  For suitable data
with mass slightly above this threshold, the analysis of Merle and
Rapha\"el~\cite{MR05,MR06} describes the
log--log blow-up regime, while Bourgain and Wang~\cite{BourgainWang97}
constructed solutions with pseudoconformal blow-up and a regular remainder.
These results concern the Euclidean equation; its continuous scaling and
pseudoconformal symmetries are not symmetries of NLS on a fixed torus. Global well-posedness of the the cubic nonlinear Schrödinger equation on $\mathbb{T}^2$ has been well studied in \cite{HK2026}.

Concentration can nevertheless survive changes in the spatial setting.
Burq, G\'erard and Tzvetkov~\cite{BGT03} constructed critical-mass
finite-time blow-up solutions in the interior of a plane domain, and their
construction also applies to the flat two-dimensional torus.
Banica~\cite{Banica04} studied the blow-up rate and possible boundary
concentration for the Dirichlet problem.
For inhomogeneous mass-critical NLS, Banica, Carles and
Duyckaerts~\cite{BCD11} constructed minimal-mass blow-up solutions under
suitable local flatness assumptions at a maximum of the nonlinear coefficient.
In dimension two, Rapha\"el and Szeftel~\cite{RaphaelSzeftel11} proved existence
and uniqueness, up to phase, at a prescribed nondegenerate maximum of the
nonlinear coefficient, with prescribed blow-up time and admissible energy.
Perturbations can also change
the rate itself: Le~Coz, Martel and Rapha\"el~\cite{LeCozMR16} obtained
critical-mass blow-up for a focusing double-power NLS with a rate determined
by the lower-power term.  These works show why small terms in a rescaled
equation must be retained when they act on a neutral direction of the
solitary wave.

The closest precedent for the mechanism considered here is the work of
Martel and Rapha\"el~\cite{MartelRaphael18}.  For the cubic focusing equation
on \(\mathbb R^2\), they constructed, for every integer \(K\ge2\), a
solution consisting asymptotically of \(K\) same-phase ground-state bubbles
whose centers approach a regular polygon and whose common scale is
asymptotic to \(1/\log t\).  The interaction changes the leading motion.
This differs from constructions with asymptotically decoupled components,
including Merle's prescribed multiple blow-up points~\cite{Merle90}, the
multi-solitary waves of Martel and Merle~\cite{MartelMerle06}, and their
supercritical counterparts constructed by C\^ote, Martel and
Merle~\cite{CMM11}.  Long-range interactions also determine the trajectories
of the Hartree two-solitons constructed by Krieger, Martel and
Rapha\"el~\cite{KMR09}.  On the torus, one concentrating bubble interacts
with its periodic translates.  The proof below keeps these exponentially
small interactions and computes their effect on the critical scaling
direction.

The construction uses a terminal problem near the prescribed concentrating
profile, followed by a compactness argument.  Its modulation argument draws on
Weinstein's work on ground states~\cite{Weinstein85} and on the blow-up
constructions discussed above.  The method is not confined to local dispersion:
Krieger, Lenzmann and Rapha\"el~\cite{KLR13} constructed critical-mass
blow-up for the half-wave equation without a pseudoconformal symmetry.
Nor is infinite-time concentration specific to NLS.
For the critical generalized Korteweg--de~Vries equation, Martel, Merle and
Rapha\"el~\cite{MMR15} constructed solutions near the soliton with polynomial
or exponential grow-up, driven by prescribed spatial tails.
Here the forcing comes from the periodic geometry, and the logarithmic
Sobolev inequality of Br\'ezis and Gallou\"et~\cite{BrezisGallouet80} is used
in the continuation argument at critical mass.  Since the gradient norm
itself diverges, the solution constructed
here has a different concentration mechanism from the bounded-\(H^1\)
energy transfers in the defocusing problem.

\subsection*{The construction and its uniqueness}

The periodic geometry produces a weak interaction even for one
concentrating profile. At rescaled period \(D=L/\lambda\), the
overlap in the mass of two neighbouring images has size
\(D^{1/2}e^{-D}\). The corresponding nonlinear overlap has size
\(D^{-1/2}e^{-D}\). Correcting the mass before computing the energy
removes the larger first-order term and leaves a negative energy of
the latter size. A small quadratic phase can compensate for it.
The resulting positive phase parameter satisfies
\[
 b\sim \sqrt{2\alpha}\,D^{-1/4}e^{-D/2},
 \qquad
 \alpha=\frac{8c_Q I_Q}{I_2},\quad
 I_Q=\int_{\mathbb R^2}e^{y_1}Q^3,\quad
 I_2=\int_{\mathbb R^2}|y|^2Q^2,
\]
where \(Q(r)\sim c_Qr^{-1/2}e^{-r}\).
In renormalized time \(dt/ds=\lambda^2\), the scale satisfies
\(D_s\sim Db\). Hence
\[
 t\sim\sqrt{2/\alpha}\,L^2D^{-11/4}e^{D/2},
\]
which explains both logarithmic terms in \eqref{intro:scale}.

To turn this calculation into an exact solution, we prescribe data at
a large scale \(D_n\), integrate backwards to a fixed scale \(D_0\),
and then let \(D_n\to\infty\). The large physical travel time makes
an ordinary Lipschitz estimate for the flow unsuitable. Instead, we
project onto the symplectic complement of the modulation directions.
The constrained Hessian is positive, and its Hamiltonian evolution
cancels in the differentiated energy. Retaining the change of the
projection gives an estimate with only a polynomial loss in \(D\).
The exponentially small source can then be integrated uniformly from
the terminal scale. Higher graph energies satisfy a tame inequality,
so the same terminal family is bounded at every fixed regularity
order on the same initial section.

There is also a uniqueness statement for the resulting asymptotic
class. We spell out the class because the leading norm asymptotic
\eqref{intro:growth} alone does not specify it.
Let \(\rho\) be the radial decaying solution of
\[
 (-\Delta+1-3Q^2)\rho=|y|^2Q/4,
\]
and define
\[
 Z_{D,b}f(y)=\sum_{j\in\mathbb Z^2}
            e^{-ib|y-Dj|^2/4}f(y-Dj),\qquad
 P_{D,b}=Z_{D,b}(Q+a(D,b)\rho),
\]
where \(a(D,b)\) is chosen to give mass \(m\) in one periodic
cell. This choice is unique in the small neighbourhood used here.
Set \(B_0(D)=\sqrt{2\alpha}D^{-1/4}e^{-D/2}\).
Here \(\mathbb T_D^2=\mathbb R^2/(D\mathbb Z^2)\), and
the pairings below are
\(\langle f,g\rangle=\operatorname{Re}\int_{\mathbb T_D^2}f\overline g\).
For a representation
\begin{equation}\label{intro:class-decomposition}
 u(t,x)=e^{i\gamma(t)}\lambda(t)^{-1}
              (P_{D(t),b(t)}+\varepsilon(t))(x/\lambda(t)),
       \qquad D=L/\lambda,
\end{equation}
we impose the real orthogonality conditions
\begin{equation}\label{intro:class-constraints}
 \langle\varepsilon,Z_{D,b}(|y|^2Q)\rangle
 =\langle\varepsilon,iZ_{D,b}\rho\rangle
 =\langle\varepsilon,iZ_{D,b}\Lambda Q\rangle=0,
 \qquad \Lambda Q=Q+y\cdot\nabla Q .
\end{equation}
These equations fix the nearby scale, phase and quadratic phase
parameter.

\begin{theorem}\label{thm:uniqueness}
Let \(u_1,u_2\) be smooth solutions of \eqref{eq:nls}, even in
each coordinate, with mass \(m\) and energy zero. Suppose that,
on positive-time tails, both admit
\eqref{intro:class-decomposition}--\eqref{intro:class-constraints}
about the same centre, with \(D(t)\to\infty\),
\[
 b/B_0(D)=1+O(D^{-1}),\qquad
 \|\varepsilon(t)\|_{H^k(\mathbb T_{D(t)}^2)}
       \le C_k B_0(D(t))/D(t)
\]
for every fixed integer \(k\ge0\).
Then there are constants \(\theta,\tau\in\mathbb R\) such that
\[
 u_1(t)=e^{i\theta}u_2(t+\tau)
\]
on their common translated lifespan.
\end{theorem}

Theorem~\ref{thm:main} produces a solution in this class.
If time and phase are normalized at the same outgoing scale section,
Theorem~\ref{thm:uniqueness} gives a single limiting datum. In
particular the normalized terminal sequence converges in every fixed
\(H^k\) without extracting a subsequence.
Corollary~\ref{rate:corollary} makes this convergence quantitative:
terminal data prescribed at scale \(N\) give an initial \(H^1\)
error bounded by \(CN^P e^{-N}\). At every fixed Sobolev order,
the corresponding solutions converge at rate \(e^{-\eta N}\),
for any \(\eta<1\), on each fixed compact time interval.

The proof of uniqueness requires more than repeating the estimate for
one orbit. At a common scale, the two solutions have slightly
different physical clocks. Their difference initially appears with
an exponentially large coefficient. We first subtract a quasistatic
correction obtained by inverting the constrained generator. This
improves the error from \(O(B_0/D)\) to a polynomial times \(B_0^2\)
at every fixed order. The exact energy constraint and differentiated
modulation equations then make the clock difference small enough.
A positive difference energy has only a polynomial propagation loss,
which is dominated by its vanishing terminal value.

Finally, a localized version of the sharp inequality implies global
continuation for any smooth datum of mass at most \(m\) and
nonpositive energy. This extends the constructed orbit in both time
directions. At the same mass and energy we give a different smooth,
nonstationary, square-symmetric solution which stays bounded in
\(H^1\). The concentrating behaviour therefore uses the terminal
construction as well as the conserved quantities.

Section~\ref{sec:profile} develops the mass-constrained profile and
computes its interaction. Section~\ref{sec:construction} proves the
modulation estimates, constructs the terminal family and obtains the
positive-time asymptotics. Section~\ref{sec:global} proves global
continuation and the bounded comparison. The uniqueness argument is
given in Section~\ref{sec:uniqueness}.

\section{A mass-constrained periodic profile}\label{sec:profile}

The interaction responsible for the construction is already visible in a
periodic sum of copies of the Euclidean ground state.  Its mass overlap is
larger than its nonlinear overlap by a factor of the separation.  We first
compute both quantities, and then impose the mass constraint before choosing
the quadratic phase.

Throughout the proof, integrals on a torus use Lebesgue measure, without
normalization.  Complex function spaces are regarded as real spaces, with
\[
 \langle f,g\rangle=\operatorname{Re}\int f\overline g.
\]
The period of the physical torus is denoted by \(L\), whereas \(D=L/\lambda\)
is the period after rescaling.  All estimates below hold for \(D\) sufficiently
large.  A constant with an index \(k\) may depend on a fixed number of
derivatives, but never on \(D\).  We use the abbreviation
\[
 \mathcal F(D)=C D^p e^{-D}
\]
only in estimates where the value of the finite exponent \(p\) is
irrelevant; it may be increased in a fixed finite collection of estimates.

When the domain is clear, \(\|f\|_p\) denotes its \(L^p\) norm.
Integrals of unperiodized profiles are over \(\mathbb R^2\);
integrals of periodic profiles are over one cell of the indicated torus.
Norms without a displayed domain follow the same convention.
The notation \(M(f)\), without a domain subscript, likewise refers
to the mass on the domain of \(f\).
For nonnegative quantities, \(A\lesssim B\) means \(A\le CB\),
and \(A\asymp B\) means that both comparisons hold, with constants
independent of \(D\). A subscript records any additional dependence.

\subsection{The ground state and the interaction integrals}

Let \(Q\) be the positive radial solution of
\[
 -\Delta Q+Q-Q^3=0\qquad\text{on }\mathbb R^2,
\]
and set
\begin{equation}\label{prof:constants}
 \begin{gathered}
 m=\int Q^2=\int|\nabla Q|^2=\frac12\int Q^4,\qquad
 I_2=\int |y|^2Q^2,\qquad c_0=\frac{I_2}{8},\\
 I_Q=\int Q^3(y)e^{y_1}\,dy,\qquad
 Q(r)=c_Qr^{-1/2}e^{-r}(1+O(r^{-1})),\\
 q=c_QI_Q,\qquad \alpha=\frac{q}{c_0},\qquad
 R(D)=D^{-1/2}e^{-D},\qquad
 B_0(D)=\sqrt{2\alpha}\,D^{-1/4}e^{-D/2}.
 \end{gathered}
\end{equation}
The integral identities follow by multiplying the equation by \(Q\) and
by \(y\cdot\nabla Q\).  We shall also use
\[
 L_+=-\Delta+1-3Q^2,\qquad L_-=-\Delta+1-Q^2,\qquad
 \Lambda f=f+y\cdot\nabla f.
\]
There is a unique radial decaying solution \(\rho\) of
\begin{equation}\label{prof:rho}
 L_+\rho=\frac{|y|^2Q}{4}.
\end{equation}
The standard ground-state spectral theory gives this assertion, the
generalized-kernel identities
\begin{equation}\label{prof:kernel}
 L_-Q=0,\qquad L_+\Lambda Q=-2Q,\qquad
 L_-(|y|^2Q)=-4\Lambda Q,
\end{equation}
and the fact that every derivative of \(\rho\) is bounded by a polynomial
times \(e^{-|y|}\); see \cite[(1.20)--(1.23)]{MartelRaphael18}.
In particular, integration by parts in \eqref{prof:rho} gives
\begin{equation}\label{prof:rho-mass}
 -2\langle\rho,Q\rangle
 =\frac14\langle |y|^2Q,\Lambda Q\rangle=-\frac{I_2}{4},
 \qquad \langle\rho,Q\rangle=c_0.
\end{equation}

The constants in \eqref{prof:constants} enter through the four nearest
periodic images.  The following computation also records the derivative
bounds needed later.

\begin{lemma}[Periodic overlaps]\label{prof:overlap}
Define
\[
 \begin{gathered}
 \Phi_D(y)=\sum_{k\in\mathbb Z^2}Q(y-Dk),\\
 A_D=\sum_{k\ne0}\int Q(y)Q(y-Dk)\,dy,\qquad
 C_D=\sum_{k\ne0}\int Q^3(y)Q(y-Dk)\,dy.
 \end{gathered}
\]
For \(j=0,1,2\),
\begin{equation}\label{prof:overlap-asymptotic}
 \begin{split}
 \partial_D^j A_D
 &=(-1)^j2qD^{1/2}e^{-D}(1+O(D^{-1})),\\
 \partial_D^j C_D
 &=(-1)^j4qD^{-1/2}e^{-D}(1+O(D^{-1})).
 \end{split}
\end{equation}
Moreover,
\begin{equation}\label{prof:unfolding}
 \int_{\mathbb T_D^2}\Phi_D^2=m+A_D,\quad
 \int_{\mathbb T_D^2}|\nabla\Phi_D|^2=m+C_D-A_D,\quad
 \int_{\mathbb T_D^2}\Phi_D^4=2m+4C_D+\mathcal R_D,
\end{equation}
where \(\mathcal R_D\ge0\) and, for every \(\eta>0\),
\[
 |\partial_D^j\mathcal R_D|\le C_{\eta,j}e^{-(2-\eta)D},
 \qquad j=0,1,2.
\]
\end{lemma}

\begin{proof}
Here \(f*g\) denotes convolution on \(\mathbb R^2\),
\((f*g)(x)=\int_{\mathbb R^2}f(x-y)g(y)\,dy\), and \(e_1=(1,0)\).
We use \(K_\nu\) for the modified Bessel function of the second kind
of order \(\nu\).
Put \(F=Q^3\) and \(P_1=F*F\).  The kernels of
\((-\Delta+1)^{-1}\) and \((-\Delta+1)^{-2}\) on \(\mathbb R^2\) are
\[
 G(r)=\frac{K_0(r)}{2\pi},\qquad
 H(r)=\frac{rK_1(r)}{4\pi}.
\]
For the second formula, differentiate the first resolvent kernel with
respect to the coefficient of its zeroth-order term.  Since \(Q=G*F\)
and \(F\) is even,
\[
 \int Q(y)Q(y-re_1)\,dy=(H*P_1)(re_1),\qquad
 \int Q^3(y)Q(y-re_1)\,dy=(G*P_1)(re_1).
\]
The required asymptotics follow from the heat-kernel formulas
\[
 G(r)=\frac1{4\pi}\int_0^\infty t^{-1}e^{-t-r^2/(4t)}\,dt,
 \qquad
 H(r)=\frac1{4\pi}\int_0^\infty e^{-t-r^2/(4t)}\,dt,
 \qquad r>0.
\]
Indeed, these are obtained by integrating the planar heat kernel
against \(e^{-t}\) and \(te^{-t}\), respectively.  Set
\(t=(r/2)e^v\), combine the contributions from \(v\) and \(-v\),
and then put \(z=2\sqrt r\sinh(v/2)\).  This gives
\[
 \begin{split}
 G(r)&=\frac{e^{-r}r^{-1/2}}{2\pi}
       \int_0^\infty e^{-z^2/2}a_0(z^2/r)\,dz,\\
 H(r)&=\frac{e^{-r}r^{1/2}}{4\pi}
       \int_0^\infty e^{-z^2/2}a_1(z^2/r)\,dz,
 \end{split}
\]
where
\[
 a_0(w)=(1+w/4)^{-1/2},\qquad
 a_1(w)=(1+w/2)(1+w/4)^{-1/2},\qquad w\ge0.
\]
Both amplitudes equal one at zero and have bounded first and second
derivatives on \([0,\infty)\).  Thus, after multiplication by
\(\sqrt{2/\pi}\), each Gaussian integral equals \(1+E_j(r)\), where
\[
 |\partial_r^k E_j(r)|\le C_k r^{-1-k},
 \qquad r\ge1,\quad j=0,1,\quad k=0,1,2.
\]
For \(k=0\), use \(|a_j(w)-1|\le Cw\); for \(k=1,2\),
differentiate under the integral and use the finite Gaussian moments.
Since \(\int_0^\infty e^{-z^2/2}\,dz=\sqrt{\pi/2}\), we obtain
\[
 G(r)=\frac{r^{-1/2}e^{-r}}{\sqrt{8\pi}}(1+O(r^{-1})),
 \qquad
 H(r)=\frac{r^{1/2}e^{-r}}{4\sqrt{2\pi}}(1+O(r^{-1})).
\]
These expansions can be differentiated twice.  Convolution with \(F\)
in the first formula shows that \(c_Q=I_Q/\sqrt{8\pi}\), while
\(\int P_1(y)e^{y_1}\,dy=I_Q^2\).  It follows that
\begin{equation}\label{prof:single-overlap}
 \begin{split}
 \partial_r^j\int Q(y)Q(y-re_1)\,dy
 &=(-1)^j\frac q2 r^{1/2}e^{-r}(1+O(r^{-1})),\\
 \partial_r^j\int Q^3(y)Q(y-re_1)\,dy
 &=(-1)^jq r^{-1/2}e^{-r}(1+O(r^{-1})),
 \qquad j\le2.
 \end{split}
\end{equation}
Here differentiation of the error is justified at the convolution level.
The number \(\epsilon>0\) used to split the integral is fixed
and sufficiently small.
Indeed, \(F,P_1\), and their fixed derivatives decay like
\(e^{-(3-\eta)|y|}\).  On \(|y|\le\epsilon r\), expansion of
\(|re_1-y|\) leaves an integrable error bounded by \(r^{-1}\) times a
polynomial in \(|y|\) and its exponential moment.  The complement is
exponentially smaller.  Near the kernel singularity one first inserts a
smooth cutoff and moves derivatives to \(P_1\).  Its derivatives there
are exponentially small, and the kernels themselves are locally
integrable.  This argument applies to each of the two derivatives in
\eqref{prof:single-overlap}.

Radiality reduces the sums to separations \(D|k|\).  The four terms
\(|k|=1\) give \eqref{prof:overlap-asymptotic}; all other terms have
\(|k|\ge\sqrt2\).  Their differentiated series converge absolutely and
are exponentially smaller.

The mass identity in \eqref{prof:unfolding} follows by unfolding one
periodic cell.  For a nonzero shift \(z\), the equation for \(Q\) gives
\[
 \int\nabla Q(y)\cdot\nabla Q(y-z)\,dy
 =\int Q^3(y)Q(y-z)\,dy-\int Q(y)Q(y-z)\,dy,
\]
which proves the gradient identity.  In the quartic expansion, terms
with all four centers equal contribute \(2m\); the terms with three
equal centers contribute \(4C_D\).  The remaining terms are positive.
Their four centers can be paired into two pairs of distinct centers,
so the sum of the distances from any integration point to those
centers is at least \(2D\).

After unfolding at one center, this lower bound gives the remainder
estimate.  More explicitly, each derivative in \(D\) inserts
derivatives of \(Q\) and polynomial weights in the lattice indices.
These weights can be absorbed into an arbitrarily small fraction of
the exponential decay.  A further small fraction is retained for
integration and summation, using
\[
 \sup_{D\ge1,\ y\in\mathbb R^2}
 \sum_{k\in\mathbb Z^2}e^{-c|y-Dk|}<\infty\qquad(c>0).
\]
The remaining exponential is bounded by \(e^{-(2-\eta)D}\).
This also supplies a common integrable majorant for the differentiated
series.
\end{proof}

We record one related estimate separately.  It distinguishes a product
of overlapping tails from an individual tail cut at the boundary of a
cell; the latter would give an unnecessarily large error.

\begin{lemma}[Localized product bounds]\label{prof:products}
For \(f_a=Q+a\rho\), set
\[
 Z_{D,b}f(y)=\sum_{k\in\mathbb Z^2}
 e^{-ib|y-Dk|^2/4}f(y-Dk),\qquad P_{D,b,a}=Z_{D,b}f_a.
\]
Suppose that \(|a|\le C D^{1/2}e^{-D}\) and
\(|b|\le C B_0(D)\).  For every fixed integer \(k\ge0\),
\begin{equation}\label{prof:cubic-products}
 \left\||P_{D,b,a}|^2P_{D,b,a}
 -\sum_{\ell\in\mathbb Z^2}
 e^{-ib|y-D\ell|^2/4}f_a(y-D\ell)^3
 \right\|_{H^k(\mathbb T_D^2)}
 \le C_k R(D).
\end{equation}
Fixed spatial and parameter derivatives of a product involving
distinct image centers are bounded by \(C D^p e^{-D}\), in the
corresponding fixed Sobolev norm or unfolded scalar integral.
\end{lemma}

\begin{proof}
Every mixed cubic term has at least two distinct centers.  For its
largest contribution, two factors are near one center and the third
is near a center at distance \(r\ge D\).  The estimate reduces to
\[
 \|Q(\cdot)^2Q(\cdot-re)\|_{H^k}\le C_k r^{-1/2}e^{-r},
 \qquad |e|=1.
\]
On \(|y|\le r/2\), the distant factor is bounded by
\(C r^{-1/2}e^{-r+|y|}\), and the square of the other two
factors times \(e^{2|y|}\) is integrable.  Near the other center,
or away from both centers, the repeated factors yield an additional
exponential gain.  The same proof applies to derivatives, since
fixed derivatives of \(Q\) have the same exponential tail.
Terms with three distinct centers have a larger total separation.
Summing the resulting bounds over the lattice gives
\eqref{prof:cubic-products}.  Replacing a factor by \(a\rho\)
introduces only a polynomial tail and the exponentially small factor
\(a\), so its contribution is smaller.

A derivative of a phase or of a center inserts a polynomial in the
local coordinates and the separation.  Integrating along the segment
between two centers produces at most another polynomial factor.
Thus fixed parameter derivatives satisfy the weaker, but sufficient,
bound \(C D^p e^{-D}\).  Absolute convergence of the lattice series
justifies taking these derivatives before integration.  Finally,
\(bD^2\to0\), and fixed spatial derivatives of the chirps preserve
all these estimates.
\end{proof}

\subsection{Exact mass and the zero-energy phase}

The definition of \(Z_{D,b}\) assigns a quadratic phase to each image
before summing.  It therefore gives a periodic function for every
real \(b\).  We now choose \(a\) so that its mass is exactly \(m\).

\begin{lemma}[Mass correction]\label{prof:mass}
There is a fixed \(\delta>0\) such that, for \(D\) large and
\(|b|\le\delta\), there is a unique small real smooth function
\(a=a(D,b)\), even in \(b\), such that
\begin{equation}\label{prof:exact-mass}
 \int_{\mathbb T_D^2}|P_{D,b,a(D,b)}|^2=m.
\end{equation}
On this fixed parameter neighborhood,
\(|a(D,b)|\le C D^{1/2}e^{-D}\), and every fixed parameter
derivative is bounded by a polynomial in \(D\) times \(e^{-D}\).
In the narrower region \(|b|\le C_*B_0(D)\), for any fixed
\(C_*>0\), it satisfies
\begin{equation}\label{prof:a-asymptotic}
 a(D,b)=-\frac{A_D}{2c_0}+O(D^p e^{-2D})
 =-\alpha D^{1/2}e^{-D}(1+O(D^{-1})).
\end{equation}
For any fixed finite collection of derivatives,
\[
 |\partial_D^j a|\le C_jD^{p_j}e^{-D},\qquad
 |\partial_ba|\le C|b|D^p e^{-D},\qquad
 |\partial_D^j\partial_b^\ell a|\le C_{j,\ell}D^{p_{j,\ell}}e^{-D}
 \quad(\ell\ge2).
\]
The first two \(D\)-derivatives of the first equality in
\eqref{prof:a-asymptotic} have the corresponding differentiated error
bounds.
\end{lemma}

\begin{proof}
Unfolding the mass gives the diagonal term
\(m+2c_0a+a^2\|\rho\|_2^2\).  Each off-diagonal term contains
the cosine of a phase difference, after pairing it with its complex
conjugate.  The absolute values of the off-diagonal terms are bounded
by the corresponding unchirped overlaps.  Thus for \(|b|\le\delta\)
their sum is \(O(D^{1/2}e^{-D}+|a|D^p e^{-D})\), and the
mass derivative in \(a\) remains bounded away from zero, uniformly on a
fixed sufficiently small \(a\)-interval.  The implicit function
theorem gives the stated extension and the bound
\(|a|\le C D^{1/2}e^{-D}\).  Fixed derivatives of these
off-diagonal terms insert only polynomial weights in the overlap
integrals, even when \(b\) lies in the fixed interval.

At \(a=b=0\), the off-diagonal sum is \(A_D\).  Taylor's formula,
the weighted-overlap proof of Lemma~\ref{prof:products}, and
\eqref{prof:rho-mass} therefore give
\[
 \int |P_{D,b,a}|^2
 =m+A_D+2c_0a+
 O\bigl(a^2+|a|D^p e^{-D}+b^2D^p e^{-D}\bigr).
\]
Conjugation gives evenness in \(b\).  In the narrower region
\(|b|\le C_*B_0(D)\), substituting the bound for \(a\) into this
expansion yields the first equality in \eqref{prof:a-asymptotic}.
The second follows from Lemma~\ref{prof:overlap}.

All derivatives claimed in the statement follow by differentiating
this exact mass equation.  Its \(a\)-derivative remains bounded away
from zero, whereas fixed derivatives of the off-diagonal sums are
bounded by the weighted estimates of Lemma~\ref{prof:products}.
In particular their first \(b\)-derivative vanishes at \(b=0\),
and their second \(b\)-derivative is \(O(D^p e^{-D})\).
This gives the extra factor \(|b|\) in \(\partial_ba\).
Repeated implicit differentiation proves the remaining bounds;
no differentiation of an unspecified asymptotic error is required.
\end{proof}

From now on, write \(P_{D,b}=P_{D,b,a(D,b)}\).  Let
\[
 E_D(v)=\frac12\int_{\mathbb T_D^2}|\nabla v|^2
       -\frac14\int_{\mathbb T_D^2}|v|^4,\qquad
 S_D(v)=E_D(v)+\frac12\int_{\mathbb T_D^2}|v|^2.
\]
We suppress the parameters when convenient, writing
\(P=P_{D,b}\) and \(a=a(D,b)\).
The prime on \(S_D\) denotes its variational derivative for the real
pairing:
\[
 S_D'(v)=-\Delta v+v-|v|^2v,\qquad
 \langle S_D'(v),h\rangle
   =\left.\frac{d}{d\tau}S_D(v+\tau h)\right|_{\tau=0},
 \quad \tau\in\mathbb R .
\]

The next lemma selects the terminal phase using the exact energy,
rather than only its leading asymptotic.

\begin{lemma}[The zero-energy profile branch]\label{prof:energy-branch}
For all sufficiently large \(D\), the equation
\[
 E_D(P_{D,b})=0
\]
has a unique solution \(b=B(D)\) in
\([B_0(D)/2,2B_0(D)]\).  It is smooth, positive, and
\begin{equation}\label{prof:B-asymptotic}
 B(D)=B_0(D)(1+O(D^{-1})),\qquad
 \frac{B'(D)}{B(D)}=-\frac12+O(D^{-1}),\qquad
 |\partial_D^jB(D)|\le C_jD^{p_j}B_0(D)
\end{equation}
for each fixed \(j\).  Uniformly for \(b\in[B_0(D)/4,4B_0(D)]\),
\begin{equation}\label{prof:energy-derivative}
 \partial_b E_D(P_{D,b})
 =2c_0b+O(|b|D^p e^{-D})\ge cB_0(D).
\end{equation}
In particular \(\lambda^{-1}P_{D,B(D)}(x/\lambda)\), with
\(\lambda=L/D\), has mass \(m\) and energy zero on the physical torus.
\end{lemma}

\begin{proof}
First consider the real profile.  The normalized periodic sum
\(\varphi_D=(m/(m+A_D))^{1/2}\Phi_D\) has mass \(m\).
The identities \eqref{prof:unfolding} give the exact formula
\[
 E_D(\varphi_D)
 =-\frac{2mA_D^2+2m(m-A_D)C_D+m^2\mathcal R_D}
         {4(m+A_D)^2}
 =-\frac12C_D+O(e^{-3D/2}).
\]
The larger term \(A_D\) cancels at first order.  The error and its
first two \(D\)-derivatives satisfy the displayed bound, by
Lemma~\ref{prof:overlap}.

The difference \(h_D=P_{D,0}-\varphi_D\) is
\(O(D^p e^{-D})\) in every fixed Sobolev norm.  Since both functions
have mass \(m\), expansion of their actions yields
\[
 E_D(P_{D,0})-E_D(\varphi_D)
 =\langle S_D'(\varphi_D),h_D\rangle+O(\|h_D\|_{H^1}^2+
 \|h_D\|_{H^1}^4).
\]
The mixed-product estimates and the mass normalization give
\(\|S_D'(\varphi_D)\|_2\le C D^p e^{-D}\).
Consequently this difference, and its fixed derivatives when
required, is \(O(D^p e^{-2D})\).

The energy is even in \(b\).  Its diagonal phase contribution is
\[
 \frac{b^2}{8}\int_{\mathbb R^2}|y|^2(Q+a\rho)^2
 =c_0b^2+O(|a|b^2).
\]
Every other \(b\)-dependent term has second derivative bounded by
\(C D^p e^{-D}\), including the dependence through \(a(D,b)\).
This follows directly by differentiating the unfolded cross terms
and the exact mass equation.  We have proved
\begin{equation}\label{prof:energy-expansion}
 E_D(P_{D,b})
 =c_0b^2-\frac12C_D+O(e^{-3D/2}+b^2D^p e^{-D}),
\end{equation}
as well as \eqref{prof:energy-derivative}.  The same expansion holds
after the \(D\)-derivatives used below.

At \(b=B_0/2\) the leading energy is negative, whereas at \(b=2B_0\)
it is positive.  The intermediate value theorem and
\eqref{prof:energy-derivative} give the unique root.
The identity \(c_0B_0^2=2qR\), together with
\eqref{prof:overlap-asymptotic}, proves the first assertion in
\eqref{prof:B-asymptotic}.  Implicit differentiation of the exact
energy equation gives
\[
 2c_0BB'=\frac12 C_D'+O(D^p e^{-3D/2}),
\]
and hence \(B'/B=-1/2+O(D^{-1})\).  Higher derivatives are obtained
from the same equation and the differentiated image sums.
They are bounded by a polynomial times \(B_0\).
Finally, critical scaling preserves mass and multiplies energy by
\(\lambda^{-2}\), so the physical assertions follow.
\end{proof}

For an unknown solution, the energy of the error must be measured
relative to the mass-constrained profile.  We record the identity
which will later keep its phase on the branch of
Lemma~\ref{prof:energy-branch}.

\begin{lemma}[Energy comparison on the constraint surface]
\label{prof:energy-comparison}
Set
\[
 T_1=Z_{D,b}(|y|^2Q),\qquad
 T_2=iZ_{D,b}\rho,\qquad T_3=iZ_{D,b}\Lambda Q,
\]
and
\[
 \mathcal H_{P}(\varepsilon)
 =S_D(P+\varepsilon)-S_D(P)-\langle S_D'(P),\varepsilon\rangle.
\]
For \(b\in[B_0/4,4B_0]\), define
\begin{equation}\label{prof:stationary-residual}
 r_{D,b}=S_D'(P_{D,b})-\frac{a+b^2}{4}T_1-bT_3.
\end{equation}
For every fixed \(k\),
\[
 \|r_{D,b}\|_{H^k}\le C_k R(D),\qquad
 \|\partial_b r_{D,b}\|_{H^k}\le C_kD^{p_k}e^{-D}.
\]
If
\[
 M(P+\varepsilon)=M(P)=m,\qquad
 \langle\varepsilon,T_j\rangle=0\ (j=1,2,3),\qquad
 E_D(P+\varepsilon)=0,
\]
and \(N=\|\varepsilon\|_{H^1}\) is small, then
\begin{equation}\label{prof:exact-energy-comparison}
 0=E_D(P_{D,b})+\mathcal H_{P_{D,b}}(\varepsilon)
                      +\langle r_{D,b},\varepsilon\rangle,
\end{equation}
and
\begin{equation}\label{prof:branch-comparison}
 |b-B(D)|\le C\,\frac{N^2+R(D)N}{B_0(D)}.
\end{equation}
\end{lemma}

\begin{proof}
For one image, conjugating by \(e^{-ib|y|^2/4}\) gives
\[
 S_D'(e^{-ib|y|^2/4}f_a)
 =e^{-ib|y|^2/4}
 \left[\frac{a+b^2}{4}|y|^2Q+ib\Lambda Q
 +iab\Lambda\rho+\frac{ab^2}{4}|y|^2\rho
 -3a^2Q\rho^2-a^3\rho^3\right].
\]
After summation, the additional term is the negative mixed cubic
sum of Lemma~\ref{prof:products}.  The displayed remaining central
terms are \(O(R)\) in every fixed Sobolev norm.  Their first
\(b\)-derivative, and the differentiated cross terms, are bounded
by a polynomial times \(e^{-D}\).  This proves the residual bounds.

The exact mass equality implies
\[
 2\langle P,\varepsilon\rangle+\|\varepsilon\|_2^2=0.
\]
Expanding the action rather than the energy therefore gives
\[
 E_D(P+\varepsilon)-E_D(P)
 =\langle S_D'(P),\varepsilon\rangle+\mathcal H_P(\varepsilon).
\]
The \(T_1,T_3\) components in
\eqref{prof:stationary-residual} vanish by the stated
orthogonalities, proving \eqref{prof:exact-energy-comparison}.
Uniform Sobolev embeddings on \(\mathbb T_D^2\), \(D\ge1\), give
\[
 |\mathcal H_P(\varepsilon)|
 \le C(N^2+N^3+N^4)\le CN^2.
\]
The embeddings follow, for example, by a cover by unit squares
with bounded overlap.  Hence
\(|E_D(P_{D,b})|\le C(N^2+RN)\).
Apply the mean value theorem between \(b\) and \(B(D)\), and use
\eqref{prof:energy-derivative}.  This proves
\eqref{prof:branch-comparison}.
\end{proof}

The scalar branch and the periodic profiles are now fixed.  In the next
section they serve as coordinates for exact solutions.  All parameter
derivatives will be taken after transport to the physical torus; no
dilation operator will be applied to an unknown periodic function.

\section{The terminal construction}\label{sec:construction}

\subsection{Coordinates, constraints, and the positive Hessian}

We work in the subspace of functions even in each coordinate.  This
symmetry fixes the center and removes the translation directions from
the coercivity estimate.  The additional symmetry under interchange
of the coordinates will be preserved by the terminal data, but is
not needed in the local argument.

It is convenient to identify the changing scaled torus with the fixed
physical torus by the isometry
\[
 I_D f(x)=\lambda^{-1}f(x/\lambda),\qquad \lambda=L/D.
\]
For an integer \(k\ge0\), define
\begin{equation}\label{mod:scaled-norm}
 \|w\|_{k,\lambda}^2
 =\sum_{|\beta|\le k}\lambda^{2|\beta|}
      \|\partial_x^\beta w\|_{L^2(\mathbb T_L^2)}^2.
\end{equation}
We write \(H^k_\lambda\) for \(H^k(\mathbb T_L^2)\) equipped
with this scaled norm.
Thus \(\|I_D f\|_{k,\lambda}=\|f\|_{H^k(\mathbb T_D^2)}\),
with the derivative definition of the integer Sobolev norm.
Write \(V_{D,b}=I_DP_{D,b}\), and let \(T_1,T_2,T_3\) be the tests
in Lemma~\ref{prof:energy-comparison}.

In this section \(P=P_{D,b}\), \(V=V_{D,b}\), and the scalar
profile quantities are evaluated at the current \(D,b\).
Subscripts \(t\) and \(s\) denote differentiation along the parameter
curve; for physical profiles and operators, the coordinate \(x\)
is held fixed, as is the argument of an operator.
The notation \(a_D,a_b\) denotes the partial
derivatives of \(a(D,b)\).

The choice of tests separates scale, phase, and quadratic phase.
The following elementary calculation gives coordinates with constants
that remain uniform as the bubble concentrates.

\begin{lemma}[Local modulation coordinates]\label{mod:coordinates}
For \(D\) large and \(|b|\le C_*B_0(D)\), a reflection-even function
sufficiently close to \(e^{i\gamma}V_{D,b}\) has unique nearby
parameters \((\log\lambda,\gamma,b)\) for which
\begin{equation}\label{mod:decomposition}
 u=e^{i\gamma}(V_{D,b}+w),\qquad
 \varepsilon=I_D^{-1}w,\qquad
 \langle\varepsilon,T_j\rangle=0\quad(j=1,2,3).
\end{equation}
The phase is understood locally modulo \(2\pi\).  The neighborhood
and inverse Jacobian bounds are uniform when distance is measured
in the scaled \(H^1\) norm.  The resulting parameters lie in the
fixed chart \(|b|<\delta\) of Lemma~\ref{prof:mass}; membership
in the smaller interval \(|b|\lesssim B_0(D)\) is a separate
condition, imposed in the bootstrap below.
\end{lemma}

\begin{proof}
Apply the implicit function theorem to
\[
 \left\langle e^{-i\gamma}u-V_{D,b},I_DT_j\right\rangle,
 \qquad j=1,2,3,
\]
on the fixed physical torus.  In the limit \(D\to\infty\), \(b\to0\),
the three rescaled profile tangents are
\[
 -\Lambda Q,\qquad iQ,\qquad -i|y|^2Q/4.
\]
Their pairings with the three limiting tests form the matrix
\begin{equation}\label{mod:coordinate-matrix}
 \begin{pmatrix}
 I_2&0&0\\
 0&c_0&-\langle |y|^2Q,\rho\rangle/4\\
 0&0&I_2/4
 \end{pmatrix}.
\end{equation}
Its determinant is \(16c_0^3>0\).  The identity
\(\langle |y|^2Q,\Lambda Q\rangle=-I_2\), and
\(\langle Q,\Lambda Q\rangle=0\), account for its triangular form.

The image terms and the derivatives of \(a(D,b)\) perturb this
matrix by \(o(1)\), while allowing \(|b|\le\delta\) adds an
\(O(\delta)\) perturbation.  Choose the fixed \(\delta\) small
enough that the matrix remains invertible.  The broad mass
extension in Lemma~\ref{prof:mass} makes the profile defined on
this entire neighborhood.  For uniformity of the implicit function theorem,
differentiate each image at its fixed physical center \(Lk\):
the derivatives in \(\log\lambda\) and \(b\) insert polynomial
weights in \((x-Lk)/\lambda\).  All resulting profiles have
uniformly bounded scaled Sobolev norms.  Lemma~\ref{prof:mass}
controls the derivatives of \(a\).  The finite collection of
pairing maps therefore has a uniform \(C^2\) bound, and its
Jacobian has a uniformly bounded inverse.  This proves the claim.
\end{proof}

The coordinates alone do not control the error.  For that purpose
we use the action Hessian
\begin{equation}\label{mod:hessian}
 \mathcal L_P z=-\Delta z+z-2|P|^2z-P^2\overline z.
\end{equation}
It is self-adjoint on the real \(L^2\) space.  Its negative and
neutral directions are removed by the mass condition and the
coordinate constraints.

A pairing between a Hessian and \(H^1\) functions is understood
as its continuous quadratic form, or the associated polarized form,
obtained by integration by parts. This convention also applies to
the physical Hessian \(\mathcal L\) below.

\begin{lemma}[Uniform constrained coercivity]\label{mod:coercivity}
For \(D\) large, \(|b|\le C_*B_0(D)\), and reflection-even
\(\varepsilon\), one has
\begin{equation}\label{mod:coercivity-projections}
 \begin{split}
 \langle\mathcal L_P\varepsilon,\varepsilon\rangle
 \ge c\|\varepsilon\|_{H^1}^2
 -C\big(&|\langle\varepsilon,Z_{D,b}Q\rangle|^2
       +|\langle\varepsilon,T_1\rangle|^2\\
       &+|\langle\varepsilon,T_2\rangle|^2\big).
 \end{split}
\end{equation}
Consequently, if \eqref{mod:decomposition} holds and
\(M(P+\varepsilon)=M(P)=m\), then, for
\(\|\varepsilon\|_{H^1}\) sufficiently small,
\begin{equation}\label{mod:relative-coercivity}
 c\|\varepsilon\|_{H^1}^2
 \le \mathcal H_P(\varepsilon)
 \le C\|\varepsilon\|_{H^1}^2.
\end{equation}
The same positive quadratic estimate holds for
\(\langle\mathcal L_P z,z\rangle\) on the linear space
\(\langle z,P\rangle=\langle z,T_j\rangle=0\), \(j=1,2,3\).
\end{lemma}

\begin{proof}
The Euclidean ground-state coercivity estimate is
\[
 \begin{split}
 \langle L_+\operatorname{Re}z,\operatorname{Re}z\rangle
 +\langle L_-\operatorname{Im}z,\operatorname{Im}z\rangle
 \ge c\|z\|_{H^1}^2-C\big(&|\langle z,Q\rangle|^2
 +|\langle z,|y|^2Q\rangle|^2\\
 &+|\langle z,i\rho\rangle|^2\big)
 \end{split}
\]
for reflection-even \(z\); the usual additional translation
pairings vanish by symmetry.  This is the standard constrained
ground-state estimate; see \cite[(1.24)]{MartelRaphael18}.

We transfer it to the torus by localization.  Choose a large fixed
radius \(r\), then \(D>8r\), and an even quadratic partition of
unity \(\chi_r^2+\zeta_r^2=1\), with \(\chi_r=1\) on the ball
of radius \(r\), \(\chi_r=0\) outside the ball of radius \(2r\),
and gradients bounded by \(C/r\).  Apply the Euclidean estimate
to \(\chi_r\varepsilon\), extended by zero.  Its pairings differ
from the periodized pairings by at most
\(o_{r\to\infty}(1)\|\varepsilon\|_2\), uniformly for \(D>8r\).
In the core, replacing \(\Phi_D\) by \(Q\) gives an exponentially
small error in \(D\).  On the support of \(\zeta_r\), the
periodic potential is \(O(e^{-cr})\), so the free quadratic form
is positive.  The localization identity loses at most
\(Cr^{-2}\|\varepsilon\|_2^2\).  First choose \(r\) large and then
\(D\) large to absorb these errors.

This proves \eqref{mod:coercivity-projections} for \(a=b=0\).
The estimates
\[
 \|Z_{D,b}f-Z_{D,0}f\|_2\le C_f|b|,\qquad
 \|P-\Phi_D\|_\infty\le C(|a|+|b|)
\]
follow from \(|e^{-ib|y|^2/4}-1|\le |b||y|^2/4\) and the
localized tails.  They perturb the quadratic form and its test
pairings by a quantity tending to zero.  This proves the general
case of \eqref{mod:coercivity-projections}.

The exact mass identity is
\begin{equation}\label{mod:mass-identity}
 \langle\varepsilon,Z_{D,b}Q\rangle
 =-\frac12\|\varepsilon\|_2^2
       -a\langle\varepsilon,Z_{D,b}\rho\rangle.
\end{equation}
Its square is bounded by
\(C(\|\varepsilon\|_{H^1}^2+a^2)\|\varepsilon\|_{H^1}^2\),
and is absorbable.  The two remaining bad pairings vanish.
Finally, \(\mathcal H_P\) is one half of the Hessian form plus
cubic and quartic terms, bounded by
\(C(\|\varepsilon\|_{H^1}^3+\|\varepsilon\|_{H^1}^4)\).
This proves \eqref{mod:relative-coercivity}.  On the linear
constraint space, the first pairing in
\eqref{mod:mass-identity} equals
\(-a\langle z,Z_{D,b}\rho\rangle\), which is absorbed in the
same way.
\end{proof}

For higher derivatives, it is useful to replace the nonlinear
mass surface by its tangent space and keep the difference as a
finite-dimensional error.  The required projection is symplectic,
rather than orthogonal.

\begin{lemma}[The symplectic projection]\label{mod:projection}
On the physical torus, put
\[
 \begin{gathered}
 J=-i,\qquad g_0=V_{D,b},\qquad g_j=I_DT_j\quad(1\le j\le3),\\
 W=\{z:\langle z,g_j\rangle=0,\ 0\le j\le3\},\qquad
 S=\operatorname{span}_{\mathbb R}\{Jg_j:0\le j\le3\}.
 \end{gathered}
\]
For \(D\) large these spaces give a direct sum \(W\oplus S\).
Let \(\Pi\) be the projection onto \(W\) along \(S\).
In operator expressions, \(I\) is the identity and \(\Pi^*\)
is the adjoint with respect to the real \(L^2\) pairing.
Then
\begin{equation}\label{mod:projection-identities}
 \Pi^2=\Pi,\qquad \Pi J=J\Pi^*,\qquad
 (\Pi J)^*=-\Pi J,\qquad \Pi\Pi_s\Pi=0.
\end{equation}
The last identity holds along any differentiable parameter curve.
The projection is uniformly bounded in every fixed scaled Sobolev
space, and
\begin{equation}\label{mod:projection-derivative}
 \|\Pi_s z\|_{k,\lambda}
 \le C_k\bigl(|\lambda_s/\lambda|+|b_s|\bigr)\|z\|_2.
\end{equation}
If \(w\) satisfies the mass and coordinate conditions above and
\(y=\Pi w\), \(n=(I-\Pi)w\), then
\begin{equation}\label{mod:normal-part}
 \|n\|_{k,\lambda}\le C_k\|w\|_2^2,\qquad
 \|w\|_{k,\lambda}\le C_k\|y\|_{k,\lambda}
\end{equation}
when \(\|w\|_2\) is small.

In \(DN(V)\), \(N\) denotes the cubic map \(z\mapsto |z|^2z\),
also written \(\mathcal N\) below, and \(D\) denotes its real
Fr\'echet derivative.

For
\begin{equation}\label{mod:physical-hessian}
 \mathcal L=-\lambda^2\Delta+1-\lambda^2 DN(V),\qquad
 DN(V)z=2|V|^2z+V^2\overline z,
\end{equation}
one has
\[
 c\|z\|_{1,\lambda}^2\le\langle\mathcal Lz,z\rangle
 \le C\|z\|_{1,\lambda}^2\qquad(z\in W).
\]
In addition, the normal space is almost invariant:
\begin{equation}\label{mod:normal-invariance}
 \|\Pi J\mathcal L n\|_{k,\lambda}
 \le C_k B_0(D)\|n\|_2\qquad(n\in S).
\end{equation}
\end{lemma}

\begin{proof}
Use the symplectic form
\(\Omega(v,z)=\langle J^{-1}v,z\rangle\).
We write
\(S^\Omega=\{v:\Omega(v,z)=0\text{ for every }z\in S\}\).
The four vectors \(Jg_j\), after rescaling, converge to
\((-iQ,-i|y|^2Q,\rho,\Lambda Q)\).
Writing \(r_0=\langle |y|^2Q,\rho\rangle\), their symplectic
matrix converges to
\[
 \begin{pmatrix}
 0&0&c_0&0\\
 0&0&r_0&-I_2\\
 -c_0&-r_0&0&0\\
 0&I_2&0&0
 \end{pmatrix}.
\]
Its determinant is \(c_0^2I_2^2>0\).  Thus \(S\) is symplectic,
\(W=S^\Omega\), and the claimed direct sum follows.
The projection formula uses the inverse of this four by four
matrix and the four pairings with \(g_j\).
Symplectic orthogonality gives
\(\Pi J=J\Pi^*\); the other algebraic identities in
\eqref{mod:projection-identities} follow by taking adjoints and
differentiating \(\Pi^2=\Pi\).

All vectors in the projection formula, their duals, and their
fixed parameter derivatives have uniformly bounded scaled
Sobolev norms.  This follows by differentiating at the fixed
physical image centers, as in Lemma~\ref{mod:coordinates}.
It proves boundedness and \eqref{mod:projection-derivative}.
The constraint pairings of \(w\) are exactly
\[
 \langle w,g_j\rangle=0\quad(1\le j\le3),\qquad
 \langle w,g_0\rangle=-\frac12\|w\|_2^2.
\]
The finite-rank formula therefore gives the first estimate in
\eqref{mod:normal-part}.  At order zero, absorb
\(\|w\|_2\le\|y\|_2+C\|w\|_2^2\), and then use the same
normal estimate at order \(k\) to obtain the second.
The coercive estimate is precisely the linear assertion of
Lemma~\ref{mod:coercivity}, transported by \(I_D\).

Write \(\mathcal L_\infty\) for the real-linear Euclidean Hessian,
\[
 \mathcal L_\infty z
   =L_+\operatorname{Re}z+iL_-\operatorname{Im}z,
 \qquad z:\mathbb R^2\to\mathbb C .
\]
For the last assertion, the Euclidean operator sends the limiting
normal basis according to
\[
 J\mathcal L_\infty(-iQ)=0,\quad
 J\mathcal L_\infty(-i|y|^2Q)=4\Lambda Q,\quad
 J\mathcal L_\infty\rho=-i|y|^2Q/4,\quad
 J\mathcal L_\infty\Lambda Q=2iQ.
\]
Every image lies in the same normal space and is killed by the
limiting projection.  The actual localized operator, spanning
vectors, and projection differ from these identities by
\(O(|b|+|a|+D^p e^{-D})\) in every fixed scaled Sobolev norm.
One obtains this bound by conjugating each localized vector
by its own chirp and then estimating the cross-image products;
Lemma~\ref{prof:products} applies to the latter.
The bound is \(O(B_0)\), and finite-dimensional norm equivalence
on \(S\) proves \eqref{mod:normal-invariance}.
\end{proof}

\subsection{The exact modulation equations}

We next differentiate the coordinates along an exact solution.
The derivatives of the test functions are part of this calculation.
Keeping them produces a cancellation between the quadratic phase
and the motion of the tests.

Let \(s\) denote renormalized time, so \(dt/ds=\lambda^2\), and set
\begin{equation}\label{mod:defects}
 h=\lambda_s/\lambda,\qquad
 \mu=h+b,\qquad \omega=\gamma_s-1,\qquad
 \nu=b_s+b^2-2b\mu-a(D,b).
\end{equation}
Thus
\begin{equation}\label{mod:parameter-equations}
 D_s=D(b-\mu),\qquad
 b_s=a-b^2+2b\mu+\nu,\qquad \gamma_s=1+\omega.
\end{equation}
These are definitions until the error estimates below are proved.

\begin{lemma}[Residual and dynamic modulation]\label{mod:modulation}
Let a smooth reflection-even solution of mass \(m\) have the
decomposition \eqref{mod:decomposition}, with
\(D\) large, \(|b|\le C_*B_0(D)\), and
\(N=\|w\|_{1,\lambda}\) sufficiently small.
Then
\begin{equation}\label{mod:estimate}
 |\mu|+|\omega|+|\nu|
 \le C\bigl[R(D)+( |b|+D^p e^{-D})N+N^2\bigr].
\end{equation}
For \(U=e^{i\gamma}V_{D,b}\), define its full scaled residual
\[
 f=\lambda^2e^{-i\gamma}(iU_t+\Delta U+|U|^2U),
 \qquad \mathcal E=I_D^{-1}f.
\]
For every fixed \(k\),
\begin{equation}\label{mod:full-residual}
 \|f\|_{k,\lambda}=\|\mathcal E\|_{H^k}
 \le C_k\bigl[R(D)+|\mu|+|\omega|+|\nu|\bigr].
\end{equation}
The modulation vector is a smooth function of
\((D,b,\varepsilon)\) on the constraint surface.  At
\(\varepsilon=0\), its value \(\mathbf m^0(D,b)\) satisfies
\[
 |\mathbf m^0|\le CR(D),\qquad
 |\partial_D^r\partial_b^\ell\mathbf m^0|
 \le C_{r,\ell}D^{p_{r,\ell}}e^{-D}
\]
for each fixed finite derivative collection.
\end{lemma}

\begin{proof}
Write \(f_a=Q+a\rho\), and, in physical coordinates,
\[
 V(s,x)=\lambda^{-1}\sum_{k\in\mathbb Z^2}
 e^{-ib|r_k|^2/4}f_a(r_k),\qquad r_k=(x-Lk)/\lambda.
\]
Then \((r_k)_s=-hr_k\).  Differentiating this expression and
using the Laplacian product rule gives
\begin{equation}\label{mod:exact-residual}
 \begin{split}
 \mathcal E={}&Z_{D,b}\left[
 -i\mu\Lambda f_a-\omega f_a+\frac{\nu}{4}|y|^2f_a
 +ia_s\rho
 +a^2\left(3Q\rho^2+\frac{|y|^2\rho}{4}\right)+a^3\rho^3
 \right]\\
 &+\left(|P|^2P-
 \sum_k e^{-ib|y-Dk|^2/4}f_a(y-Dk)^3\right).
 \end{split}
\end{equation}
To check the coefficient of the quadratic term, before use of
\eqref{prof:rho} it is \(b_s-2bh-b^2=a+\nu\).
The term \(a|y|^2Q/4\) cancels the linear contribution of
the profile deformation.  Thus the quadratic correction left in
\eqref{mod:exact-residual} has the stated sign.

The derivative of the exact mass correction is
\begin{equation}\label{mod:as}
 a_s=Db\,a_D+(a-b^2)a_b+(-Da_D+2ba_b)\mu+a_b\nu.
\end{equation}
Consequently the residual has the exact affine form
\begin{equation}\label{mod:affine-residual}
 \mathcal E=\mathcal E_0+\mu R_\mu+\omega R_\omega+\nu R_\nu,
\end{equation}
where
\begin{equation}\label{mod:residual-directions}
 \begin{split}
 R_\mu&=-iZ_{D,b}\Lambda f_a
           +i(-Da_D+2ba_b)Z_{D,b}\rho,\\
 R_\omega&=-P,\qquad
 R_\nu=Z_{D,b}(|y|^2f_a/4)+ia_bZ_{D,b}\rho.
 \end{split}
\end{equation}
Here \(\mathcal E_0\) is obtained from
\eqref{mod:exact-residual} by setting the three defects to zero
and replacing \(a_s\) by \(Db\,a_D+(a-b^2)a_b\).
Lemmas~\ref{prof:mass} and \ref{prof:products} give
\begin{equation}\label{mod:E0}
 \|\mathcal E_0\|_{H^k}\le C_kR(D),\qquad
 \|\partial_b\mathcal E_0\|_{H^k}\le C_kD^{p_k}e^{-D}.
\end{equation}
Indeed \(a_s\) in this expression is bounded by a polynomial
times \(e^{-3D/2}\), and the terms quadratic in \(a\) are
bounded by a polynomial times \(e^{-2D}\).
Only the mixed cubic sum has size \(R(D)\).
The profiles in \eqref{mod:residual-directions} have uniformly
bounded fixed Sobolev norms.

We turn to the differentiated orthogonalities.  To avoid confusion
with the physical vectors \(g_j\), denote the three Euclidean
test functions by
\[
 \psi_1=|y|^2Q,\qquad \psi_2=i\rho,\qquad
 \psi_3=i\Lambda Q,\qquad T_j=Z_{D,b}\psi_j.
\]
Their physical transports \(t_j=I_DT_j\) satisfy
\(\langle w,t_j\rangle=0\).  Direct differentiation at the
fixed centers gives
\begin{equation}\label{mod:Kj}
 I_D^{-1}(t_j)_s
 =K_j=-hZ_{D,b}\Lambda\psi_j
       -i(a+b^2+\nu)Z_{D,b}(|y|^2\psi_j/4).
\end{equation}
The cubic Taylor remainder in scaled coordinates is
\[
 \mathcal N_2(P,\varepsilon)
 =\overline P\,\varepsilon^2+2P|\varepsilon|^2+
                                      |\varepsilon|^2\varepsilon.
\]
Subtracting the profile equation from NLS, on the fixed torus,
gives
\begin{equation}\label{mod:w-equation}
 iw_s=\mathcal Lw+\omega w-f-I_D\mathcal N_2(P,\varepsilon).
\end{equation}
Now differentiate \(\langle w,t_j\rangle=0\), use the real
self-adjointness of the Hessian, and only then rescale the
integrals.  The result is the exact identity
\begin{equation}\label{mod:projection-equation}
 \langle\mathcal E,iT_j\rangle
 =\langle\varepsilon,\mathcal L_P(iT_j)+K_j+\omega iT_j\rangle
                 -\langle\mathcal N_2(P,\varepsilon),iT_j\rangle.
\end{equation}
All integrations by parts here are on the fixed periodic torus.

We spell out the cancellation in its rows.  The defect-independent
part of \(K_j\) is
\[
 K_j^0=bZ_{D,b}\Lambda\psi_j
       -i(a+b^2)Z_{D,b}(|y|^2\psi_j/4).
\]
For a single image and any localized complex \(z\),
\[
 \mathcal L_{e^{-ib|y|^2/4}f_a}
       (e^{-ib|y|^2/4}z)
 =e^{-ib|y|^2/4}\left[
 (-\Delta+1-2f_a^2)z-f_a^2\overline z
 +ib\Lambda z+\frac{b^2|y|^2}{4}z\right].
\]
Taking \(z=i\psi_j\), the term \(ib\Lambda(i\psi_j)\)
cancels the first term of \(K_j^0\).  The two \(b^2\) terms
also cancel.  Therefore, with
\[
 C_1=-4T_3,\qquad C_2=-T_1/4,\qquad C_3=2Z_{D,b}Q,
\]
the actual periodic row errors satisfy
\begin{equation}\label{mod:row-cancellation}
 W_j:=\mathcal L_P(iT_j)+K_j^0-C_j,\qquad
 \|W_j\|_2+\|\partial_bW_j\|_2\le C D^p e^{-D}.
\end{equation}
For completeness, after the cancellation the single-image
expression is
\[
 Z_{D,b}\left[
 (-\Delta+1-2f_a^2)(i\psi_j)
 -f_a^2\overline{i\psi_j}
 -\frac{ia|y|^2}{4}\psi_j\right].
\]
At \(a=0\), \eqref{prof:kernel} and \eqref{prof:rho}
give \(C_j\) exactly.  Replacing \(Q\) by \(f_a\) costs
\(O(a)\); the remaining periodic terms are products with
distinct centers and are controlled by
Lemma~\ref{prof:products}.  Their fixed \(b\)-derivatives have
the same polynomial-exponential bound.  This proves
\eqref{mod:row-cancellation}.

We can now solve the finite system explicitly.  In
\(K_j+\omega iT_j=K_j^0+\mu K_j^\mu+
\omega K_j^\omega+\nu K_j^\nu\), the coefficients are
\[
 K_j^\mu=-Z_{D,b}\Lambda\psi_j,\quad
 K_j^\omega=iT_j,\quad
 K_j^\nu=-iZ_{D,b}(|y|^2\psi_j/4).
\]
Set
\begin{equation}\label{mod:finite-matrix}
 A_{j\ell}(D,b,\varepsilon)
 =\langle R_\ell,iT_j\rangle
       -\langle\varepsilon,K_j^\ell\rangle,
 \qquad \ell\in\{\mu,\omega,\nu\}.
\end{equation}
At \(\varepsilon=0\), this matrix converges to
\eqref{mod:coordinate-matrix}, so its inverse is uniformly
bounded for small \(N\).  The exact right-hand side is
\begin{equation}\label{mod:finite-numerator}
 \begin{split}
 F_j(D,b,\varepsilon)={}&
 \langle\varepsilon,W_j\rangle
 -\langle\mathcal N_2(P,\varepsilon),iT_j\rangle
 -\langle\mathcal E_0,iT_j\rangle\\
 &-\mathbf1_{j=3}
 \left(\|\varepsilon\|_2^2+
                   2a\langle\varepsilon,Z_{D,b}\rho\rangle\right).
 \end{split}
\end{equation}
Here \(\mathbf1_{j=3}\) is one when \(j=3\) and zero otherwise.
We write \(F=(F_1,F_2,F_3)^{\mathsf T}\) for the right-hand side vector.
The final term is \eqref{mod:mass-identity}; it is essential
because \(C_3=2Z_{D,b}Q\) is not one of the imposed tests.
The terms \(C_1,C_2\) vanish by orthogonality.
Thus \eqref{mod:projection-equation} is exactly
\begin{equation}\label{mod:smooth-map}
 A(D,b,\varepsilon)(\mu,\omega,\nu)^{\mathsf T}
       =(F_1,F_2,F_3)^{\mathsf T}.
\end{equation}

Uniform \(H^1\)-to-\(L^4,L^6\) embeddings give
\[
 |\langle\mathcal N_2(P,\varepsilon),iT_j\rangle|
 \le C(N^2+N^3).
\]
Equations \eqref{mod:E0}, \eqref{mod:row-cancellation}, and
\eqref{mod:finite-numerator} actually give the stronger bound
\[
 |\mu|+|\omega|+|\nu|
 \le C\bigl(R(D)+D^p e^{-D}N+N^2\bigr).
\]
In particular \eqref{mod:estimate} follows.  Substitution in
\eqref{mod:affine-residual} proves
\eqref{mod:full-residual}.

Finally, \eqref{mod:finite-matrix} and
\eqref{mod:finite-numerator} define smooth maps on an \(H^1\)
neighborhood, even without imposing the constraints.  Their
restriction to the constraint surface is the vector obtained
from the PDE.  At \(\varepsilon=0\), the numerator consists
only of the pairings of \(-\mathcal E_0\) with \(iT_j\).
Differentiating these unfolded scalar pairings and the uniformly
invertible matrix gives the stated bounds for \(\mathbf m^0\).
The derivative estimates of Lemmas~\ref{prof:mass} and
\ref{prof:products} justify every such fixed differentiation.
\end{proof}

The same calculation holds for \(H^1\) solutions on their local
lifespan by approximation: their time derivative lies in \(H^{-1}\),
whereas all physical tests are smooth.  For the construction below
we first use smooth terminal solutions, so every differentiated
identity is classical.  The estimates then pass to the limiting
solution by the compactness argument.

We shall need one consequence of the explicit finite system in the
uniqueness argument.  It is useful to record it now, while the
terms in that system are visible.

\begin{corollary}[Modulation derivatives]\label{mod:lipschitz}
On the constraint surface let
\(\mathbf m(D,b,\varepsilon)=(\mu,\omega,\nu)\).
For \(\|\varepsilon\|_{H^1}\le N_*\ll1\) and
\(b\in[B_0/4,4B_0]\), its smooth extension given by
\eqref{mod:smooth-map} satisfies
\begin{equation}\label{mod:map-derivatives}
 \|D_\varepsilon\mathbf m\|
       \le C(B_0+N_*),\qquad
 |\partial_b\mathbf m|\le C(D^p e^{-D}+N_*).
\end{equation}
Here \(D_\varepsilon\mathbf m\) is the real Fr\'echet derivative,
and its norm is the operator norm from reflection-even
\(H^1(\mathbb T_D^2)\) to \(\mathbb R^3\), with the Euclidean norm
on \(\mathbb R^3\). It is taken on the stated smooth extension
with \(D,b\) fixed; in \(\partial_b\mathbf m\), \(D,\varepsilon\)
are held fixed.
Consequently
\[
 |\mathbf m(D,b_1,\varepsilon_1)-\mathbf m(D,b_2,\varepsilon_2)|
 \le C(B_0+N_*)\|\varepsilon_1-\varepsilon_2\|_{H^1}
       +C(D^p e^{-D}+N_*)|b_1-b_2|.
\]
Also
\[
 |\mathbf m(D,b,\varepsilon)-\mathbf m^0(D,b)|
       \le C(B_0N+N^2),\qquad N=\|\varepsilon\|_{H^1}.
\]
\end{corollary}

\begin{proof}
The matrix in \eqref{mod:finite-matrix} and its first
\(b,\varepsilon\) derivatives are uniformly bounded.
The derivative in \(\varepsilon\) of the quadratic and cubic
pairings in \eqref{mod:finite-numerator} is
\(O(N_*+N_*^2)\), and their \(b\)-derivative is
\(O(N_*^2+N_*^3)\).
The row and residual bounds
\eqref{mod:row-cancellation}, \eqref{mod:E0} control the other
terms.  Differentiating \(A^{-1}F\) adds a bounded matrix
times \(|\mathbf m|\), which is
\(O(R+B_0N_*+N_*^2)\).
Since any fixed polynomial times \(e^{-D}\) is \(o(B_0)\),
these estimates give \eqref{mod:map-derivatives}.  The mean
value theorem in the smooth extension proves the comparison
bound; setting the second remainder to zero gives the last
assertion.
\end{proof}

The remainder equation, its constraints, and the full forcing are
now explicit.  We next combine them in a positive graph energy.
The identity \(\Pi J=J\Pi^*\) will cancel the principal Hamiltonian
part, leaving only coefficients that vary on the slow modulation
scale.

\subsection{An energy adapted to the moving constraints}

The modulation estimates control the three parameters, but they do not by
themselves control the error over the very long time intervals under
consideration. For this purpose we use the positive Hessian on the
constrained space. Applying ordinary derivatives to the equation would
destroy its cancellation. Powers of the constrained Hamiltonian generator
preserve that cancellation and provide the required higher-order estimates.

All operators in this subsection act on the fixed physical torus. Write
\[
 \mathcal N(z)=|z|^2z,\qquad
 \mathcal L=-\lambda^2\Delta+1-\lambda^2D\mathcal N(V),\qquad
 \mathcal A=\Pi J\mathcal L\Pi,\qquad J=-i.
\]
Here \(D\mathcal N(V)w=2|V|^2w+V^2\overline w\) is real linear,
and \(\Pi\) is the projection constructed in
Lemma~\ref{mod:projection}. We use real \(L^2\) pairings throughout.
Set \(y=\Pi w\), \(n=(I-\Pi)w\), and
\[
 \kappa=|\lambda_s/\lambda|+|b_s|,\qquad
 \mathcal R(V,w)=\mathcal N(V+w)-\mathcal N(V)-D\mathcal N(V)w .
\]
If \(U=e^{i\gamma}V\), define the residual in physical coordinates by
\[
 f=\lambda^2e^{-i\gamma}
       (iU_t+\Delta U+\mathcal N(U)).
\]
The estimates of the preceding subsection give, for every fixed integer
\(r\ge0\),
\begin{equation}\label{con:residual-bound}
 \|f\|_{r,\lambda}\le C_r(R+|\mu|+|\nu|+|\omega|).
\end{equation}

For a scalar or matrix-valued coefficient \(A(x)\), its scaled
\(W^{r,\infty}\) norm means
\[
 \|A\|_{W^{r,\infty}_\lambda}
   =\sum_{|\beta|\le r}\lambda^{|\beta|}
                         \|\partial_x^\beta A\|_{L^\infty(\mathbb T_L^2)}.
\]
For a real-linear coefficient, this convention is applied to its
matrix in real and imaginary parts.

We first record the operator estimates used below. Their uniformity as
\(\lambda\) tends to zero is important.

\begin{lemma}\label{con:graph}
For every fixed integer \(r\ge0\),
\begin{align}
 \|\mathcal A q\|_{r,\lambda}
   &\le C_r\|q\|_{r+2,\lambda},&
 \|\mathcal A_s q\|_{r,\lambda}
   &\le C_r\kappa\|q\|_{r+2,\lambda},\label{con:operator-bounds}\\
 |\langle\mathcal L_s q,q\rangle|
   &\le C\kappa\|q\|_{1,\lambda}^2.&\notag
\end{align}
If \(q\) belongs to \(W=\operatorname{Ran}\Pi\), then
\begin{equation}\label{con:elliptic}
 \|q\|_{r+2,\lambda}
 \le C_r\bigl(\|\mathcal A q\|_{r,\lambda}+\|q\|_2\bigr).
\end{equation}
Consequently, for \(j\ge1\),
\begin{equation}\label{con:graph-energy}
 X_j^2=\sum_{a=0}^j
       \langle\mathcal L\mathcal A^a y,\mathcal A^a y\rangle
 \asymp_j\|y\|_{2j+1,\lambda}^2
 \asymp_j\|w\|_{2j+1,\lambda}^2
\end{equation}
whenever the scaled \(H^1\) error is sufficiently small and the exact
mass and modulation constraints hold.
\end{lemma}

\begin{proof}
The Fourier multiplier \(-\lambda^2\Delta+1\) satisfies the usual elliptic
estimate in the scaled norms, with a constant independent of \(\lambda\).
Multiplication by \(\lambda^2D\mathcal N(V)\) is uniformly bounded in
every fixed scaled \(W^{r,\infty}\) norm. Interpolation absorbs its
highest-order contribution into the elliptic estimate, leaving a multiple
of \(\|q\|_2\). Moreover,
\((I-\Pi)J\mathcal Lq\) is a finite sum of smooth functions multiplied
by pairings against \(J\mathcal Lq\). Moving the derivatives in those
pairings onto the smooth dual vectors shows that its scaled \(H^r\)
norm is bounded by \(C_r\|q\|_2\). This proves
\eqref{con:elliptic}.

The derivative of the principal coefficient is
\(-2(\lambda_s/\lambda)\lambda^2\Delta\). All derivatives of the
potential and of the finite-rank projection are bounded by \(\kappa\)
times the corresponding scaled coefficient bounds. This follows by
differentiating the localized profiles at fixed lattice centres.
In particular, the calculation never differentiates an unknown periodic
function by a nonperiodic dilation field. These observations prove
\eqref{con:operator-bounds}.

Each \(\mathcal A^a y\) lies in \(W\). Coercivity on \(W\), followed
by repeated use of \eqref{con:elliptic}, proves the first equivalence
in \eqref{con:graph-energy}. The reverse estimate follows from
\eqref{con:operator-bounds}. Finally, the exact mass constraint and the
three linear orthogonality conditions imply
\begin{equation}\label{con:normal}
 \|n\|_{r,\lambda}\le C_r\|w\|_2^2 .
\end{equation}
At order zero one first absorbs the quadratic term in
\(\|w\|_2\le\|y\|_2+C\|w\|_2^2\). Substitution back into
\eqref{con:normal} then gives equivalence of the higher norms.
Only the order-zero smallness is needed; there is no regularity-dependent
smallness condition on \(w\).
\end{proof}

We now differentiate these energies. The following identity keeps track of
the change of the constrained space, which would otherwise introduce
uncontrolled terms.

\begin{lemma}\label{con:energy-evolution}
The projected error satisfies
\begin{equation}\label{con:projected}
 y_s=\mathcal A y+\Pi_s y+g,\qquad
 g=\Pi_s n+\Pi J\mathcal Ln+\omega\Pi Jw
          -\lambda^2\Pi J\mathcal R(V,w)-\Pi Jf .
\end{equation}
In particular \(g\in W\). On any interval where the scaled \(H^3\)
error is small,
\begin{equation}\label{con:low-energy}
 |(X_1)_s|\le
 C(\kappa+|\omega|)X_1+CX_1^2+C\|f\|_{3,\lambda}.
\end{equation}
For every fixed \(j\ge1\) one also has the tame estimate
\begin{equation}\label{con:high-energy}
 |(X_j)_s|\le
 C_j(\kappa+|\omega|+\|w\|_{3,\lambda})X_j
       +C_j\|f\|_{2j+1,\lambda}.
\end{equation}
The square-root inequalities are understood in their integral form at
zeros of \(X_j\).
\end{lemma}

\begin{proof}
Substitution of \(u=e^{i\gamma}(V+w)\) into the equation gives
\[
 w_s=J\mathcal Lw+\omega Jw
                 -\lambda^2J\mathcal R(V,w)-Jf.
\]
Applying \(\Pi\) proves \eqref{con:projected}. To check its range,
differentiate \(\Pi^2=\Pi\); the identity
\(\Pi_s(I-\Pi)=\Pi\Pi_s(I-\Pi)\) takes care of its first term.

In the following recursion, \(a\) is a nonnegative integer
index, \(\mathcal A^a\) is repeated composition at the current
time, and \(\mathcal A^0=I\). The symbols \(g_a\) denote the
forcing terms at these graph levels.
Put \(z_a=\mathcal A^a y\). For every \(a\ge0\),
\begin{equation}\label{con:recursion}
 (z_a)_s=\mathcal A z_a+\Pi_s z_a+g_a,\qquad
 g_0=g,\qquad
 g_{a+1}=\Pi\mathcal A_s z_a+\mathcal A g_a .
\end{equation}
Indeed \(\mathcal A\Pi_s z_a=0\), because
\(\Pi\Pi_s\Pi=0\). Differentiating \(\Pi\mathcal A=\mathcal A\)
also gives
\((I-\Pi)\mathcal A_s z_a=\Pi_s z_{a+1}\).
These two identities prove the recursion. Expansion of its right-hand
side yields
\[
 \|g_a\|_{1,\lambda}
 \le C_a\kappa\|y\|_{2a+1,\lambda}
                  +C_a\|g\|_{2a+1,\lambda}.
\]
Only one time derivative of an operator appears in each term, and that
operator has spatial order two. Thus no additional derivatives of the
unknown error are needed.

For \(z\in W\), the identity \(\Pi J=J\Pi^*\) implies
\[
 \langle\mathcal Lz,\mathcal A z\rangle
       =\langle\mathcal Lz,\Pi J\mathcal Lz\rangle=0.
\]
Differentiating \eqref{con:graph-energy}, using
\eqref{con:recursion}, and applying the form estimate
\(|\langle\mathcal Lq,r\rangle|
 \le C\|q\|_{1,\lambda}\|r\|_{1,\lambda}\), we obtain
\begin{equation}\label{con:form-derivative}
 |(X_j^2)_s|\le
 C_j\kappa X_j^2+C_jX_j\|g\|_{2j+1,\lambda}.
\end{equation}
The pairings here are quadratic-form pairings on \(H^1\).
For the initially smooth solutions the differentiations are classical;
the displayed bound subsequently uses only \(2j+1\) derivatives.

After rescaling, the nonlinear remainder consists of
\(2P|\varepsilon|^2+\overline P\varepsilon^2+
|\varepsilon|^2\varepsilon\).
Uniform Sobolev and product inequalities on the expanding tori give
\[
 \|\lambda^2\mathcal R(V,w)\|_{3,\lambda}
 \le C(\|w\|_{3,\lambda}^2+\|w\|_{3,\lambda}^3)
\]
and, for every fixed \(r\ge3\),
\[
 \|\lambda^2\mathcal R(V,w)\|_{r,\lambda}
 \le C_r(\|w\|_{3,\lambda}+\|w\|_{3,\lambda}^2)
                    \|w\|_{r,\lambda}.
\]
For example, apply the tame product estimate to
\(\varepsilon^2\) and \(|\varepsilon|^2\varepsilon\), and then
multiply by \(P\), whose fixed-order coefficient norms are bounded.
The constants are uniform by localization to unit squares with bounded
overlap. The \(L^2\) part of the inhomogeneous norm controls the low
frequencies.

The normal terms in \(g\) obey the same tame estimate by
\eqref{con:normal}. The term \(\mathcal Ln\) uses that estimate at
order \(r+2\), which is legitimate because \(n\) is finite rank.
The phase term is bounded by \(C_r|\omega|\|w\|_{r,\lambda}\).
Combining these bounds with \eqref{con:form-derivative} proves both
claims, using a regularized square root if necessary.
\end{proof}

\subsection{A uniform family with a prescribed terminal scale}

We can now construct exact solutions. Fix a large terminal scale \(D_n\)
and set \(\lambda_n=L/D_n\). At a temporary terminal time prescribe
\begin{equation}\label{con:terminal}
 u_n^{\mathrm{term}}=V_{D_n,B(D_n)} .
\end{equation}
By the choice of \(B\), this datum has exactly mass \(m\) and energy
zero. Its error is zero. It is smooth and invariant under the square
symmetries. Local existence follows from the elementary contraction
argument in \(H^3(\mathbb T_L^2)\): the free Schr\"odinger group is
unitary and the cubic map is locally Lipschitz on this algebra.
Smoothness persists as long as the \(H^3\) solution exists, by the tame
product estimate.

The time coordinate will be fixed only after the solution has reached a
common scale. This avoids comparing the large additive travel times of
different terminal solutions.

On an interval where the scale increases, \(t_n(D)\) denotes
the time at which the \(n\)-th solution reaches scale \(D\).
An orbit quantity with argument \(D\), such as \(w_n(D)\) or
\(b_n(D)\), is composed with this inverse scale function.
Its subscript \(D\) denotes the total derivative in that
parametrization. When considering one member of the family,
we omit the index \(n\).

\begin{proposition}\label{con:terminal-family}
There are \(D_0\) and \(K\), independent of \(D_n\), such that the
solution from \eqref{con:terminal} can be continued backwards to the
section \(D=D_0\). Throughout this interval,
\begin{equation}\label{con:uniform-low}
 \|w_n(D)\|_{3,\lambda}\le K\frac{B(D)}D,\qquad
 \frac{b_n(D)}{B(D)}=1+O(D^{-2}),\qquad
 D_s\asymp DB(D)>0.
\end{equation}
The constants in the middle estimate are independent of \(n\).
\end{proposition}

\begin{proof}
Starting from the zero terminal error, impose the open bounds
\[
 D\ge D_0,\qquad \tfrac12B(D)<b<2B(D),\qquad
 X_1<K_1B(D)/D.
\]
Choose \(K_1\) first and \(D_0\) afterwards. By graph coercivity and
Lemma~\ref{mod:modulation},
\[
 |\mu|+|\nu|+|\omega|\le CR,\qquad
 D_s=D(b-\mu)\asymp DB,
\]
where the constants can be made independent of \(K_1\):
the contributions \(b\|w\|_{1,\lambda}/R\) and
\(\|w\|_{1,\lambda}^2/R\) are bounded by
\(CK_1/D\) and \(CK_1^2/D^2\), respectively.
The parameter equations and \eqref{con:residual-bound} then give
\(\kappa+|\omega|\le CB\) and
\(\|f\|_{3,\lambda}\le CB^2\).
Dividing \eqref{con:low-energy} by \(D_s\), we find
\[
 |(X_1)_D|\le \frac CD X_1+
        \frac C{DB}X_1^2+C\frac BD,\qquad X_1(D_n)=0.
\]
Backward Gronwall therefore gives
\[
 X_1(D)\le C\int_D^{D_n}
  (r/D)^C\left(\frac{X_1(r)^2}{rB(r)}+\frac{B(r)}r\right)\,dr.
\]
For every fixed \(a\) and \(p\),
\begin{equation}\label{con:tail}
 \int_D^\infty(r/D)^a B(r)r^{-p}\,dr
       \le C_{a,p,D_0}B(D)D^{-p},\qquad D\ge D_0.
\end{equation}
This follows directly from
\(B(r)\asymp r^{-1/4}e^{-r/2}\).
For the present fixed exponent we take \(D_0\) large enough that the
constant is absolute. Substitution of the bootstrap bound yields
\[
 X_1(D)\le C_*\left(1+\frac{K_1^2}{D_0^2}\right)\frac{B(D)}D .
\]
First take \(K_1>4C_*\), and then increase \(D_0\). This improves
the error bound by a factor of two.

Energy conservation improves the branch bound without an additional
shooting parameter. Lemma~\ref{prof:energy-comparison}, applied with
\(E(u_n)=0\), gives
\[
 |b-B(D)|\le
 C\frac{\|w\|_{1,\lambda}^2+R\|w\|_{1,\lambda}}{B(D)}
 \le C_{K_1}B(D)D^{-2}.
\]
Increasing \(D_0\) again keeps \(b\) strictly inside the factor-two
interval. This proves the claimed branch estimate as well.

For a fixed \(n\), all scales lie in the compact interval
\([D_0,D_n]\). The profile bound and the improved error estimate
give a finite physical \(H^3\) bound there. The local continuation
criterion rules out a finite endpoint before the solution reaches
\(D_0\), while the strict improvements rule out a bootstrap exit.
Finally, \(D_t=D_s/\lambda^2\) has a positive lower bound on this
compact scale interval. The backwards crossing occurs in finite time.
Absorbing the graph-norm equivalence into \(K\) proves the proposition.
\end{proof}

The preceding argument fixes \(D_0\) once and for all. Higher derivatives
will change the constants in the estimates, but will not change this
section or the sequence of initial data.

\begin{proposition}\label{con:all-orders}
For every fixed integer \(k\ge0\), the same terminal family satisfies
\begin{equation}\label{con:uniform-high}
 \|w_n(D)\|_{k,\lambda}\le C_k B(D)/D,
          \qquad D_0\le D\le D_n ,
\end{equation}
with a constant independent of \(n\).
\end{proposition}

\begin{proof}
It suffices to prove the claim at every odd order \(2j+1\ge3\).
The low-order bound has already been established, so the tame estimate
\eqref{con:high-energy} is linear in \(X_j\). The explicit residual has
norm at most \(C_j B^2\) at each fixed order. Thus
\[
 |(X_j)_D|\le C_j(D^{-1}+D^{-2})X_j+C_jB/D,
            \qquad X_j(D_n)=0.
\]
The integrating factor is bounded by a fixed power of \(r/D\)
times a constant depending on \(j,D_0\).
Equation~\eqref{con:tail} proves \(X_j(D)\le C_jB(D)/D\).
For large \(j\), that integral still converges at the original \(D_0\):
its polynomial exponent is fixed while its tail is exponential.
There is no need to enlarge \(D_0\) with \(j\).
Lower orders follow by interpolation.
\end{proof}

\subsection{Passage to one smooth orbit and its asymptotic form}

Translate each terminal solution so that its crossing of \(D=D_0\)
occurs at \(t=0\), and multiply it by a constant phase so that
\(\gamma_n(0)=0\). These operations preserve the equation and its
conserved quantities. The exact physical clock is
\begin{equation}\label{con:clock-integral}
 t_n(D)=L^2\int_{D_0}^D
       \frac{dr}{r^3(b_n(r)-\mu_n(r))}.
\end{equation}
The asymptotic expansion of \(B\) and
Proposition~\ref{con:terminal-family} imply, uniformly for \(D\le D_n\),
\begin{equation}\label{con:clock}
 t_n(D)=\sqrt{2/\alpha}\,L^2D^{-11/4}e^{D/2}
                              (1+O(D^{-1})).
\end{equation}
To see the coefficient, the integrand in \eqref{con:clock-integral}
is
\((2\alpha)^{-1/2}r^{-11/4}e^{r/2}(1+O(r^{-1}))\).
One integration by parts gives the factor two. The integral over a
fixed neighbourhood of \(D_0\) is negligible in this expansion.
In particular, the terminal physical times tend to infinity.

At the common section the scale is fixed. Proposition~\ref{con:all-orders}
therefore bounds the same sequence \(u_n(0)\) in every \(H^k\).
Take a subsequence converging strongly in \(H^2\) and weakly in \(H^3\).
Its limit \(u_0\) belongs to every \(H^k\), since the corresponding
higher norms of this sequence are bounded. Interpolation with those
bounds gives strong convergence in each fixed \(H^k\), with the same
limit.

On any fixed physical interval \([0,T]\), formula~\eqref{con:clock}
keeps the scales uniformly away from zero. The physical \(H^3\) norms
are consequently bounded there. Local \(H^2\) well-posedness and
continuous dependence, applied successively on finitely many subintervals,
give
\[
 u_n\longrightarrow u\quad\hbox{in }C([0,T];H^2).
\]
Uniqueness identifies these limits for different \(T\).
The uniform \(H^3\) bound and its continuation criterion give a
positive-time global \(H^3\) solution. The tame higher-order estimates
give \(u\in C([0,\infty);H^k)\) for every \(k\).
In particular there is one smooth datum, rather than a separate
choice of datum for each regularity order.

The parameters converge on compact time intervals. One may either use
their defining implicit equations and strong convergence, or first use
the bounds on their time derivatives and then identify the limit by the
same equations. Their estimates and
\eqref{con:uniform-high} pass to the limit. Strong \(H^2\) convergence
at the common section preserves the mass and energy, so
\(M(u)=m\), \(E(u)=0\).
The square symmetries also pass to \(u_0\) and are preserved by
uniqueness.

It remains to express the conclusion in physical variables. The exact
clock~\eqref{con:clock} implies, with \(\ell=\log(t/L^2)\),
\begin{equation}\label{con:scale-refined}
 D=2\ell+\frac{11}{2}\log\ell+c_*+
             O\!\left(\frac{\log\ell}{\ell}\right),
 \qquad c_*=\frac92\log2+\log\alpha .
\end{equation}
Indeed taking logarithms gives
\(\ell=D/2-(11/4)\log D+\frac12\log(2/\alpha)+O(D^{-1})\);
substitution of \(D=2\ell+O(\log\ell)\) gives
\eqref{con:scale-refined}.

Let
\[
 \mathcal Q_\lambda(x)=\lambda^{-1}
              \sum_{j\in\mathbb Z^2}Q((x-Lj)/\lambda).
\]
For each fixed \(k\), removing the chirp and the mass correction from
\(V\) costs at most \(C_k\lambda^{-k}(|b|+|a|)\).
The physical error is at most
\(C_k\lambda^{-k}B(D)/D\). Both quantities tend to zero exponentially
in \(D\), after multiplication by a fixed power of \(D\). Thus
\begin{equation}\label{con:profile-limit}
 \|u(t)-e^{i\gamma(t)}\mathcal Q_{\lambda(t)}\|_{H^k}
                 \longrightarrow0\qquad(t\to\infty).
\end{equation}

For completeness, use the Euclidean Fourier transform
\(\widehat Q(\xi)=\int e^{-ix\cdot\xi}Q(x)\,dx\).
The torus Fourier coefficient of \(\mathcal Q_\lambda\) at \(n\)
is \(\lambda L^{-2}\widehat Q(2\pi\lambda n/L)\).
Writing \(h=2\pi\lambda/L\), Parseval gives
\[
 \lambda^{2s}\|\mathcal Q_\lambda\|_{\dot H^s(\mathbb T_L^2)}^2
 =\frac{h^2}{(2\pi)^2}
   \sum_{n\in\mathbb Z^2}|hn|^{2s}|\widehat Q(hn)|^2
 \longrightarrow\|Q\|_{\dot H^s(\mathbb R^2)}^2 .
\]
The rapid decay of \(\widehat Q\) controls the tails of the Riemann
sum. For \(s>0\), the inhomogeneous norm has the same leading term.
Together with \eqref{con:profile-limit} and
\eqref{con:scale-refined}, this proves the positive-time assertions
of Theorem~\ref{thm:main}.

There is also a useful phase normalization. Fix any time \(t_*\) in
the modulation interval. Since \(\omega=O(R)\) and
\(D_s\asymp DB\),
\[
 \frac{d}{dD}
 \left(\gamma(t(D))-\int_{t_*}^{t(D)}\lambda(r)^{-2}\,dr\right)
             =\frac{\omega}{D_s}=O(B/D).
\]
Its right-hand side is integrable. Hence for some real \(\gamma_+\),
\begin{equation}\label{con:phase-limit}
 \gamma(t)=\gamma_++\int_{t_*}^t\lambda(r)^{-2}\,dr+O(B(D)/D).
\end{equation}
The scale in this formula is the actual modulation scale.
In particular \(\gamma(t)\sim4tL^{-2}[\log(t/L^2)]^2\), but
this leading relative approximation alone does not determine
the oscillatory factor in \eqref{con:profile-limit}.

\section{Global continuation at nonpositive energy}\label{sec:global}

The construction in Section~\ref{sec:construction} gives a smooth solution
for all positive times. We next prove that this same solution extends
to every finite negative time. The argument is independent of the
construction. Its ingredients are the sharp Euclidean inequality and a
localized moment whose derivative is controlled by the moment itself.
In particular, we do not need a classification of minimal-mass blow-up
solutions on the torus.

We write \(B_r(a)\) for the ball of radius \(r\) about \(a\)
in the flat metric on \(\mathbb T_L^2\). For \(r<L/2\), this ball
is identified with a Euclidean ball; expressions such as \(|x-p|\)
near \(p\) use these flat coordinates.

\subsection{A localized energy inequality}

Recall the sharp inequality
\begin{equation}\label{glob:euclidean-gn}
 \|v\|_{L^4(\mathbb R^2)}^4
 \le\frac2m\|v\|_{L^2(\mathbb R^2)}^2
                  \|\nabla v\|_{L^2(\mathbb R^2)}^2
\end{equation}
from \cite[Corollary~2.1]{Weinstein83}. Its coefficient can be retained on the flat
torus if an appropriate lower-order error is kept.

\begin{lemma}\label{glob:localized}
Fix \(p\in\mathbb T_L^2\). There are smooth real functions
\(\psi_p,W_p\ge0\) such that \(\psi_p\) vanishes only at \(p\),
\(\psi_p(x)\asymp |x-p|^2\) near \(p\),
\[
 |\nabla\psi_p|^2\le C\psi_p,\qquad W_p\le C\psi_p,
\]
and
\begin{equation}\label{glob:energy-local}
 E(v)\ge\frac12\left(1-\frac{M(v)}m\right)\|\nabla v\|_2^2
          -\frac{M(v)}{2m}\int W_p|v|^2
\end{equation}
for every \(v\in H^1(\mathbb T_L^2)\). In particular,
\[
 E(v)\ge-C\int\psi_p|v|^2\qquad\hbox{if }M(v)\le m.
\]
The constants can be chosen uniformly in \(p\).
\end{lemma}

\begin{proof}
Choose a finite smooth quadratic partition of unity
\(\sum_j\chi_j^2=1\), with each \(\chi_j\) supported in a flat
coordinate chart. Arrange that \(\chi_0=1\) near \(p\), and set
\(W_p=\sum_j|\nabla\chi_j|^2\).
This function vanishes near \(p\). Choose \(\psi_p\) smooth,
quadratic near \(p\), and strictly positive elsewhere. Compactness
gives the two pointwise bounds in the statement.

Minkowski's inequality and \eqref{glob:euclidean-gn}, applied to
each cutoff product extended by zero to \(\mathbb R^2\), give
\[
 \|v\|_4^2
 =\Big\|\sum_j|\chi_jv|^2\Big\|_2
 \le\sum_j\|\chi_jv\|_4^2
 \le\sqrt{\frac2m}\sum_j
             \|\chi_jv\|_2\|\nabla(\chi_jv)\|_2.
\]
Cauchy--Schwarz and the identity
\[
 \sum_j\|\nabla(\chi_jv)\|_2^2
       =\|\nabla v\|_2^2+\int W_p|v|^2
\]
therefore yield
\[
 \|v\|_4^4\le\frac{2M(v)}m
       \left(\|\nabla v\|_2^2+\int W_p|v|^2\right).
\]
This proves \eqref{glob:energy-local}. Translating the construction
on the flat torus makes its constants independent of \(p\).
\end{proof}

The same inequality applied after a periodic gauge transformation controls
the local mass current. This is the step where the sign of the energy
is used.

\begin{lemma}\label{glob:moment}
Let \(u\) be a smooth solution with \(0<M(u)\le m\) and \(E(u)\le0\).
For each fixed \(p\), the moment
\[
 V_p(t)=\int\psi_p|u(t)|^2
\]
satisfies
\begin{equation}\label{glob:moment-ode}
 |V_p'(t)|\le C V_p(t).
\end{equation}
In particular \(V_p(t)\) cannot tend to zero at a finite endpoint
of the lifespan.
\end{lemma}

\begin{proof}
For a fixed time write
\[
 A=\int|\nabla\psi_p|^2|u|^2,\qquad
 J_p=\operatorname{Im}\int\nabla\psi_p\cdot\nabla u\,\overline u .
\]
For every real \(a\), the function \(e^{ia\psi_p}u\) is periodic
and has the same mass as \(u\). Lemma~\ref{glob:localized} and
expansion of its kinetic energy give
\[
 0\le E(e^{ia\psi_p}u)+CV_p
       =E(u)+CV_p+aJ_p+\tfrac12a^2 A .
\]
The discriminant is nonpositive, so
\[
 J_p^2\le2A(E(u)+CV_p)\le C V_p^2 .
\]
The last step uses \(E(u)\le0\) and \(A\le CV_p\).
The local mass identity is \(V_p'=2J_p\), proving
\eqref{glob:moment-ode}. At a regular time \(V_p>0\), since
the nonzero \(L^2\) function \(u\) cannot be supported at the single
zero of \(\psi_p\). Gronwall supplies a positive lower bound on
each finite time interval.
\end{proof}

The calculation remains valid for an \(H^3\) solution: its time
derivative belongs to \(H^1\), which justifies the differentiated
moment and the integrations by parts. We shall use this observation
in the continuation argument.

\subsection{Concentration at a hypothetical finite endpoint}

We need only a mass concentration statement, which can be obtained
directly from \eqref{glob:euclidean-gn}. This avoids importing
Euclidean dynamical information into the periodic problem.

\begin{lemma}\label{glob:concentration}
Let \((v_n)\) be a sequence in \(H^1(\mathbb T_L^2)\).
Suppose \(M(v_n)=M>0\), \(E(v_n)=E_*\) is fixed, and
\(\|\nabla v_n\|_2\to\infty\). Then, for every sufficiently small
fixed \(r>0\),
\[
 \liminf_{n\to\infty}
       \sup_{a\in\mathbb T_L^2}\int_{B_r(a)}|v_n|^2\ge m .
\]
Thus this situation is impossible if \(M<m\). If \(M=m\), a
subsequence has
\[
 |v_n|^2\,dx\rightharpoonup m\delta_p
\]
for some \(p\in\mathbb T_L^2\).
Here \(\delta_p\) is the unit Dirac measure at \(p\), and the
arrow denotes convergence against every continuous test function
on \(\mathbb T_L^2\).
\end{lemma}

\begin{proof}
Take a smooth cutoff \(0\le\chi\le1\) supported in a flat ball
of radius \(r\), and apply \eqref{glob:euclidean-gn} to
\(\chi(\cdot-a)v_n\). Its mass is at most
\(\mu_r(v_n)=\sup_a\int_{B_r(a)}|v_n|^2\).
Integrating over the centre \(a\) gives
\[
 \|\chi\|_4^4\|v_n\|_4^4
 \le\frac{2\mu_r(v_n)}m
   \left(\|\chi\|_2^2\|\nabla v_n\|_2^2
                       +\|\nabla\chi\|_2^2M\right).
\]
The cross term disappears because
\(\int\chi\nabla\chi=0\).
At fixed energy,
\(\|v_n\|_4^4/\|\nabla v_n\|_2^2\to2\).
After division and passage to the limit, we obtain
\[
 \liminf_n\mu_r(v_n)\ge
           m\,\frac{\|\chi\|_4^4}{\|\chi\|_2^2}.
\]
Cutoffs approaching the indicator of the ball make this ratio
arbitrarily close to one. The limit in \(n\) is taken before
changing the cutoff, so its gradient norm causes no difficulty.

If \(M=m\), choose successively radii \(r_j\downarrow0\), indices
\(n_j\), and centres \(a_j\) so that the mass in \(B_{r_j}(a_j)\)
is at least \(m-1/j\). A subsequence of \(a_j\) converges to
some \(p\). The full mass lies asymptotically in any neighbourhood
of \(p\), proving the weak convergence of measures.
\end{proof}

These two local facts imply a continuation result which is useful beyond
the particular solution constructed here.

\begin{proposition}\label{glob:continuation}
Every \(H^3(\mathbb T_L^2)\) datum with \(0<M\le m\) and
\(E\le0\) has a solution global in both time directions.
If the datum is smooth, the solution is smooth at every finite time.
\end{proposition}

\begin{proof}
First, boundedness of \(H^1\) on a finite time interval prevents
\(H^3\) breakdown. Indeed, the periodic version of the logarithmic Sobolev
inequality \cite[Lemma~2]{BrezisGallouet80} and the tame cubic estimate give
\[
 \|u\|_\infty^2
 \le C_K\bigl(1+\log(2+\|u\|_{H^3})\bigr),\qquad
 \left|\frac{d}{dt}\|u\|_{H^3}\right|
       \le C\|u\|_\infty^2\|u\|_{H^3}
\]
whenever \(\|u\|_{H^1}\le K\). Equivalently,
\(\log(2+\|u\|_{H^3})\) satisfies a linear Gronwall inequality.
It cannot diverge in finite time. These estimates follow first
for smooth solutions and extend to \(H^3\) by approximation.

Suppose that a maximal \(H^3\) solution had a finite endpoint \(T\).
The preceding observation provides times \(t_n\to T\) with
\(\|\nabla u(t_n)\|_2\to\infty\).
Lemma~\ref{glob:concentration} excludes this if \(M<m\).
If \(M=m\), it gives a point \(p\) and a subsequence for which
\(V_p(t_n)\to0\). This contradicts Lemma~\ref{glob:moment}.
The same proof applies to either endpoint of the lifespan.

On each compact time interval the \(H^3\) norm is finite.
The tame estimates then propagate every higher Sobolev order of
the datum. The equation supplies the time derivatives, giving
joint smoothness in time and space.
\end{proof}

The solution of Section~\ref{sec:construction} has \(M=m\), \(E=0\),
and a smooth datum. Proposition~\ref{glob:continuation} therefore
extends that very solution to all real times. Cauchy uniqueness
identifies the extension with the already constructed positive-time
orbit. This completes the global assertion in Theorem~\ref{thm:main}.
No conclusion about its asymptotic behaviour as \(t\to-\infty\)
is needed here.

\subsection{A bounded comparison at the same mass and energy}

The zero-energy constraint does not itself force the growth in
Theorem~\ref{thm:main}. The following explicit example makes this
distinction precise, even within the square-symmetric class.

\begin{proposition}\label{glob:bounded-example}
There is a smooth, nonstationary, square-symmetric solution \(v\)
with \(M(v)=m\), \(E(v)=0\), and
\[
 \sup_{t\in\mathbb R}\|v(t)\|_{H^1(\mathbb T_L^2)}<\infty .
\]
\end{proposition}

\begin{proof}
Applying \eqref{glob:euclidean-gn} to
\(e^{-|x|^2/2}\) gives \(m\le4\pi\).
For \(0<a<1/2\), put
\[
 v_a(x)=\frac{\sqrt m}{L\sqrt{1+a^2}}
 \left(1+a\cos\frac{4\pi x_1}{L}
         +a\cos\frac{4\pi x_2}{L}\right).
\]
For a periodic function \(F\), write
\(\langle F\rangle=L^{-2}\int_{\mathbb T_L^2}F(x)\,dx\)
for its spatial average.
The averages of
\(S=\cos(4\pi x_1/L)+\cos(4\pi x_2/L)\)
are \(\langle S\rangle=\langle S^3\rangle=0\),
\(\langle S^2\rangle=1\), and \(\langle S^4\rangle=9/4\).
Hence \(M(v_a)=m\), and \(E(v_a)=0\) precisely when \(z=a^2\)
satisfies
\[
 (32\pi^2-9m/4)z^2+(32\pi^2-6m)z-m=0.
\]
Both nonconstant coefficients are positive. The polynomial is
negative at zero and positive at \(1/4\), since its value there is
\(10\pi^2-169m/64>0\). Its unique positive root therefore gives
the required \(a\).

The datum has period \(L/2\) in both coordinates. On that smaller
torus its mass is \(m/4\) and its energy is zero. The partition
argument in Lemma~\ref{glob:localized} gives, without localizing
at a prescribed point,
\[
 \|f\|_4^4\le\frac{2M(f)}m\|\nabla f\|_2^2
                   +C(L/2)^{-2}M(f)^2 .
\]
Since \(M(f)=m/4\), energy conservation bounds its gradient
uniformly. The resulting smooth global solution lifts to the
original torus, and uniqueness preserves its square symmetries
and additional half-period translations.

Finally, \(v_a\) cannot be a standing-wave profile:
the term \(v_a^3\) has a nonzero pure third harmonic in \(x_1\),
whereas \(\Delta v_a-\omega v_a\) has only its constant and
first coordinate harmonics. Thus no real \(\omega\) satisfies
\(\Delta v_a-\omega v_a+v_a^3=0\).
\end{proof}

The additional translations distinguish this example from the solution
of Theorem~\ref{thm:main}: an invariant mass measure under those
translations cannot converge to a single Dirac mass.
Proposition~\ref{glob:bounded-example} concerns boundedness in
\(H^1\), and makes no assertion of uniform bounds at higher
Sobolev orders.

\section{Uniqueness of the outgoing orbit}\label{sec:uniqueness}

We now show that the terminal construction selects a single orbit, up to
the two scalar symmetries of time translation and constant phase.  The
class in which uniqueness holds includes the full remainder estimates,
not only the leading rate of growth of the Sobolev norms.

Throughout this section the concentration centre is fixed at the origin,
and all functions are even in each coordinate.  We retain the profile,
modulation conditions, and notation of Section~\ref{sec:profile}.  Thus
\(\lambda=L/D\), \(P=P_{D,b}\),
\(V=I_DP=\lambda^{-1}P(\cdot/\lambda)\), and
\[
 u=e^{i\gamma}(V+w),\qquad
 \varepsilon=I_D^{-1}w.
\]
The map \(I_D\) is an isometry from the dimensionless \(L^2\) space to
the physical \(L^2\) space, and identifies the dimensionless \(H^k\)
norm with an equivalent physical scaled norm \(\|\cdot\|_{k,\lambda}\).
We write \(B\) for the positive zero-energy profile branch and put
\[
 B_0(D)=\sqrt{2\alpha}\,D^{-1/4}e^{-D/2},\qquad
 R(D)=D^{-1/2}e^{-D}.
\]
In particular \(B=B_0(1+O(D^{-1}))\) and \(R\asymp B^2\).
Constants and finite polynomial exponents below may depend on the
indicated Sobolev order.  Every occurrence of \(D^p\) denotes a fixed
finite power, and these powers may be increased when a finite collection
of estimates is combined.

We also use \(F=C D^p e^{-D}\) for a generic bound of the
same kind as \(\mathcal F(D)\) in Section~\ref{sec:profile}.

For the proof of Theorem~\ref{thm:uniqueness}, consider two solutions as follows.
Let \(u_1,u_2\) be smooth solutions with the same centre and coordinate
reflections, with \(M(u_j)=m\) and \(E(u_j)=0\).  Suppose that each has,
on a positive-time tail, the profile decomposition above with
\begin{equation}\label{uniq:class}
 \begin{gathered}
 D_j(t)\longrightarrow\infty,\qquad
 \frac{b_j}{B(D_j)}=1+O(D_j^{-1}),\\
 \|w_j(D)\|_{k,\lambda}\le C_{j,k}\frac{B(D)}D
 \quad\text{for every fixed integer }k\ge0.
 \end{gathered}
\end{equation}
The constants and the starting sections of the two tails may differ.

The local modulation equations imply \(D_s=D(b-\mu)>0\) on a sufficiently
large tail in \eqref{uniq:class}.  We may therefore compare solutions at
their respective crossings of the same value of \(D\).  This comparison
does not identify their physical or renormalized clocks.

\subsection{Differentiating the energy and modulation constraints}

We first record the quantitative dependence of the chirp and modulation
parameters on the remainder.  All quantities in the next two lemmas are
dimensionless and are evaluated at a fixed \(D\).  Write
\[
 \mathcal N(v)=|v|^2v,\qquad
 \mathcal S(P)=(-\Delta+1)P-\mathcal N(P).
\]
The three tests are
\[
 T_1=Z_{D,b}(|y|^2Q),\qquad
 T_2=iZ_{D,b}\rho,\qquad
 T_3=iZ_{D,b}\Lambda Q.
\]
The exact mass and modulation constraints are
\begin{equation}\label{uniq:constraints}
 2\langle P,\varepsilon\rangle+\|\varepsilon\|_2^2=0,
 \qquad \langle\varepsilon,T_j\rangle=0\quad(1\le j\le3).
\end{equation}
Here and below all pairings are real.

At fixed \(D\), the abbreviations \(P_j=P_{D,b_j}\),
\(P_B=P_{D,B(D)}\), and \(T_j(B)=T_j(D,B(D))\) suppress only
the displayed profile parameters.

\begin{lemma}[Differentiated energy constraint]\label{uniq:energy-map}
Suppose \(b_j\in[B/2,2B]\), the pairs \((b_j,\varepsilon_j)\) satisfy
\eqref{uniq:constraints} and have zero energy, and
\(\|\varepsilon_j\|_{H^1}\le N_*\le K B/D\), where \(K\) is fixed.
For sufficiently large \(D\), depending on \(K\),
\begin{equation}\label{uniq:b-difference}
 |b_1-b_2|\le
 C\frac{N_*+R}{B}\|\varepsilon_1-\varepsilon_2\|_{H^1}.
\end{equation}
The zero-energy constraint is a smooth positive-branch graph over the
transverse remainder coordinates.
\end{lemma}

\begin{proof}
Parameter differentiation of the absolutely convergent image sums gives,
at every fixed order,
\begin{equation}\label{uniq:profile-derivatives}
 \|\partial_b^rP\|_{H^k}
 +\sum_{j=1}^3\|\partial_b^rT_j\|_{H^k}\le C_{k,r},\qquad
 |a_b|\le C|b|D^pe^{-D}.
\end{equation}
Further fixed derivatives of \(a\) are bounded by polynomial multiples
of \(e^{-D}\).  Each differentiation of an image phase inserts a
polynomial weight in the distance to its own centre.  The diagonal
integrals are uniformly bounded, and the differentiated cross integrals
retain their exponential decay.  We abbreviate the stationary remainder \(r_{D,b}\) by \(r_b\).
Its subscript records the parameter; differentiation is written
as \(\partial_b r_b\).
Consequently the remainder
\begin{equation}\label{uniq:static-remainder}
 r_b=\mathcal S(P)-\frac{a+b^2}{4}T_1-bT_3
\end{equation}
satisfies
\begin{equation}\label{uniq:r-derivatives}
 \|r_b\|_{H^k}\le C_kR,\qquad
 \|\partial_b r_b\|_{H^k}\le C_kD^{p_k}e^{-D}.
\end{equation}
For instance differentiation of the central term \(ab\) costs \(a\),
which has this latter size.  These estimates follow by differentiating
the profile identities themselves.

Let \(E_D\) denote the dimensionless energy, and let
\(\mathfrak H_P(\varepsilon)\) be its relative action,
\[
 \mathfrak H_P(\varepsilon)
 =\frac12\int(|\nabla\varepsilon|^2+|\varepsilon|^2)
 -\frac14\int\bigl(|P+\varepsilon|^4-|P|^4
                  -4\operatorname{Re}(|P|^2P\overline\varepsilon)\bigr).
\]
Equal mass and \eqref{uniq:constraints} give the exact identity
\begin{equation}\label{uniq:energy-identity}
 0=E_D(P)+\mathfrak H_P(\varepsilon)
                 +\langle r_b,\varepsilon\rangle.
\end{equation}
In particular the two leading terms removed in
\eqref{uniq:static-remainder} have zero pairing with the remainder.

The profile energy is even in \(b\).  Its diagonal chirp contribution is
\(b^2\int |y|^2(Q+a\rho)^2/8\); every other second \(b\) derivative
has exponentially small overlap coefficients.  Hence, with
\(F=C D^pe^{-D}\),
\begin{equation}\label{uniq:energy-b-derivative}
 \partial_bE_D(P)=2c_0b+O(|b|F)\ge cB.
\end{equation}
If \(X=\|\varepsilon_1-\varepsilon_2\|_{H^1}\), expansion of the
quartic relative action and the uniform embeddings on the expanding
tori give
\[
 |\mathfrak H_{P_1}(\varepsilon_1)
       -\mathfrak H_{P_2}(\varepsilon_2)|
 \le C N_*X+C N_*^2|b_1-b_2|,
\]
while \eqref{uniq:r-derivatives} gives
\[
 |\langle r_{b_1},\varepsilon_1\rangle
       -\langle r_{b_2},\varepsilon_2\rangle|
 \le CRX+FN_*|b_1-b_2|.
\]
Subtract \eqref{uniq:energy-identity}, use
\eqref{uniq:energy-b-derivative}, and absorb
\((CN_*^2+FN_*)|b_1-b_2|\) into \(cB|b_1-b_2|\).
This proves \eqref{uniq:b-difference} without interpolating between two
points of the nonlinear constraint surface.

For clarity, a common smooth chart for that surface can also be given
explicitly.  At the profile root \(b=B(D)\), let
\(h_0=P_B\) and \(h_j=T_j(B)\), and let \(X_D\) be their common
real \(L^2\) orthogonal complement.  Their Gram matrix is uniformly
invertible.  We give \(X_D\cap H^1(\mathbb T_D^2)\) the \(H^1\) norm.
Derivatives in \(\eta\) are real Fr\'echet derivatives on this
space. In \(D_\eta\varepsilon-I\), \(I\) is its inclusion into
\(H^1\); the unlabelled norms of scalar \(\eta\)-derivatives below
are the corresponding operator norms.
For \(\eta\in X_D\), write
\[
 \varepsilon(b,\eta)=\eta+\sum_{\ell=0}^3c_\ell(b,\eta)h_\ell
\]
and impose \eqref{uniq:constraints}.  The finite-dimensional implicit
function theorem gives
\[
 \varepsilon(b,0)=0,\quad
 \|\varepsilon\|_{H^1}\asymp\|\eta\|_{H^1},\quad
 \|\partial_b\varepsilon\|_{H^1}\le C\|\eta\|_{H^1},
\]
and
\[
 \|D_\eta\varepsilon-I\|_{H^1\to H^1}
 \le C(|b-B|+\|\eta\|_{H^1}).
\]
The corrections have uniformly smooth finite-rank profiles.  On this
chart the left side of \eqref{uniq:energy-identity}, denoted by
\(\mathcal G_D(b,\eta)\), satisfies
\[
 \partial_b\mathcal G_D
 =2c_0b+O(|b|F+N_*^2+FN_*),\qquad
 \|D_\eta\mathcal G_D\|\le C(N_*+R).
\]
It therefore selects a unique smooth branch with
\[
 |b(D,\eta)-B(D)|\le C\frac{N_*^2+RN_*}{B},\qquad
 \|D_\eta b(D,\eta)\|\le C\frac{N_*+R}{B}.
\]
Existence follows by changing the sign of \(\mathcal G_D\) across a
sufficiently large multiple of the first displayed interval; its
positive \(b\) derivative gives uniqueness.
\end{proof}

The energy constraint controls the quadratic phase. We next compare
the remaining modulation parameters using the finite system derived
in Section~\ref{sec:construction}.

\begin{lemma}[Differentiated modulation map]\label{uniq:modulation-map}
On the constraint surface the exact modulation parameters
\(\mathbf m=(\mu,\omega,\nu)\) are restrictions of a smooth map of
\((D,b,\varepsilon)\).  In the setting of
Lemma~\ref{uniq:energy-map},
\begin{equation}\label{uniq:modulation-difference}
 |\mathbf m_1-\mathbf m_2|
 \le C(B+N_*)\|\varepsilon_1-\varepsilon_2\|_{H^1}
       +C(N_*+F)|b_1-b_2|.
\end{equation}
Its pure-profile value \(\mathbf m^0(D,b)=\mathbf m(D,b,0)\) obeys
\begin{equation}\label{uniq:pure-modulation}
 |\mathbf m^0|\le CR,\qquad
 |\partial_D^r\partial_b^\ell\mathbf m^0|
 \le C_{r,\ell}D^{p_{r,\ell}}R
\end{equation}
for each fixed finite collection of derivatives.  Moreover
\begin{equation}\label{uniq:modulation-homogeneous}
 |\mathbf m-\mathbf m^0|
 \le C(B\|\varepsilon\|_{H^1}+\|\varepsilon\|_{H^1}^2).
\end{equation}
\end{lemma}

\begin{proof}
The exact finite system is given in
\eqref{mod:finite-matrix}--\eqref{mod:finite-numerator}.  Its uniformly
invertible matrix and explicit numerator define a smooth extension off
the constraint surface, as proved in Corollary~\ref{mod:lipschitz}.
Since \(B\asymp B_0\), that corollary gives
\eqref{uniq:modulation-difference} and
\eqref{uniq:modulation-homogeneous} directly.  In particular the
\(N_*|b_1-b_2|\) term is retained until the refined remainder estimate
makes it harmless.  The extension includes the exact quadratic mass
term; it is not obtained by replacing equal mass with linear
orthogonality to \(Q\).

At \(\varepsilon=0\), the numerator of the finite system consists only
of the projections of \(-\mathcal E_0\).  Its size is \(O(R)\).
Differentiating these explicit pairings and the inverse matrix gives
\eqref{uniq:pure-modulation}: each fixed parameter derivative of a
cross-image term costs only a fixed polynomial in \(D\), while its
exponential factor is unchanged.  The diagonal terms depend on \(D\)
through the mass correction and its derivatives, and have the same
polynomial-exponential bounds.  The differentiated estimates in
Section~\ref{sec:profile} justify these operations before taking any
size bound.  For a \(D\) derivative one unfolds the cell integrals;
equivalently, the derivative of an explicit moving profile is
\(I_D^{-1}\partial_D(I_Dg)\), with the physical coordinate fixed.
This differentiates localized functions about their fixed physical
image centres, and never applies a nonperiodic dilation operator to
an unknown remainder.
\end{proof}

\subsection{A quasistatic improvement of the remainder}

We next improve the estimate in \eqref{uniq:class}.  This step removes
the large clock coefficient that would remain in a direct comparison
at the original remainder scale.

Let \(g_0=V\) and \(g_j=I_DT_j\) for \(1\le j\le3\), and recall
\[
 W=\{v:\langle v,g_j\rangle=0, 0\le j\le3\},\qquad
 S=\operatorname{span}\{Jg_0,Jg_1,Jg_2,Jg_3\}.
\]
The projection \(\Pi\) has range \(W\) and kernel \(S\).  Set
\begin{equation}\label{uniq:operators}
 J=-i,\qquad
 \mathcal L=-\lambda^2\Delta+1-\lambda^2D\mathcal N(V),\qquad
 \mathcal A=\Pi J\mathcal L\Pi.
\end{equation}
In particular \(\Pi J=J\Pi^*\), and the quadratic form of
\(\mathcal L\) is uniformly positive on \(W\).

\begin{lemma}[Inverse on the constrained space]\label{uniq:inverse}
For every fixed integer \(q\ge0\), \(h\in W\cap H^q\) has a unique
preimage \(r\in W\cap H^{q+2}\) under \(\mathcal A\), with
\begin{equation}\label{uniq:inverse-bound}
 \|r\|_{q+2,\lambda}\le C_q\|h\|_{q,\lambda}.
\end{equation}
The inverse depends smoothly on the profile parameters.
\end{lemma}

\begin{proof}
Solve, on \(W\cap H^1\), the coercive variational problem
\[
 \langle\mathcal Lr,v\rangle=\langle J^{-1}h,v\rangle
 \qquad(v\in W\cap H^1).
\]
The right side is bounded on the scaled \(H^1\) space, so
Lax--Milgram gives \(\|r\|_{1,\lambda}\le C\|h\|_2\).
The weak equation implies
\[
 \mathcal Lr-J^{-1}h=\sum_{j=0}^3c_jg_j.
\]
Testing against the four smooth basis vectors, moving the derivatives
in \(\mathcal L\) onto those vectors, and inverting their Gram matrix
gives \(|c|\le C(\|r\|_2+\|h\|_2)\).
The scaled elliptic estimate for \(-\lambda^2\Delta+1\), followed
by induction through the uniformly smooth potential, proves
\eqref{uniq:inverse-bound}.  Applying \(\Pi J\) to the weak equation
gives \(\mathcal Ar=h\), since \(Jg_j\in\ker\Pi\).
Conversely a kernel vector satisfies
\(\mathcal Lr\in W^\perp\), and hence has zero positive quadratic
form.  It must vanish.  Differentiating the variational problem together
with the four constraints proves smooth dependence, with the same
elliptic estimates at each fixed order.
\end{proof}

The inverse lets us remove the leading, slowly varying part of the
source. The remainder then gains a second exponential factor.

\begin{proposition}[Improved outgoing remainder]\label{uniq:refined}
Every solution in the class \eqref{uniq:class} satisfies
\begin{equation}\label{uniq:two-exponential}
 \|w(D)\|_{k,\lambda}\le C_kD^{p_k}B(D)^2
 \qquad\text{for every fixed integer }k\ge0
\end{equation}
on a sufficiently large tail.
\end{proposition}

\begin{proof}
Write \(y=\Pi w\) and \(n=(I-\Pi)w\).  Exact mass gives
\(\|n\|_{q,\lambda}\le C_q\|w\|_2^2\) for every fixed \(q\).
The projected equation from Section~\ref{sec:construction} is
\begin{equation}\label{uniq:projected-single}
 y_s=\mathcal Ay+\Pi_s y+g,
\end{equation}
where, writing
\(\mathcal R_N(V,w)=\mathcal N(V+w)-\mathcal N(V)-D\mathcal N(V)w\),
\begin{equation}\label{uniq:single-g}
 g=\Pi_s n+\Pi J\mathcal Ln+\omega\Pi Jw
       -\lambda^2\Pi J\mathcal R_N(V,w)-\Pi Jf.
\end{equation}
Here \(f=I_D\mathcal E\) is the full scaled physical profile residual,
including the actual modulation defects.  Each term in \(g\) belongs
to \(W\).

Replace \(\mathbf m\) in this residual by
\(\mathbf m^0(D,b)\) from Lemma~\ref{uniq:modulation-map}, and call
the resulting explicit function \(f^0\).  Define
\[
 g^0(D,b)=-\Pi Jf^0(D,b)\in W.
\]
The differentiated profile estimates and
\eqref{uniq:pure-modulation} give
\begin{equation}\label{uniq:g0-bounds}
 \|g^0\|_{q,\lambda}\le C_qB^2,\qquad
 \|\partial_Dg^0\|_{q,\lambda}
       +\|\partial_bg^0\|_{q,\lambda}\le C_qD^{p_q}B^2.
\end{equation}
These derivatives are in fixed physical coordinates.  This defines an
algebraic pure-profile source; it does not assert that \(w=0\) is a
solution.

If \(N_q=\|w\|_{q,\lambda}\), then for each fixed \(q\ge3\),
\begin{equation}\label{uniq:g-homogeneous}
 \|g-g^0\|_{q,\lambda}\le C_q B N_q.
\end{equation}
Indeed the two normal terms in \eqref{uniq:single-g} are bounded by
\(C_q(1+\kappa)N_1^2\), where
\(\kappa=|\lambda_s/\lambda|+|b_s|\lesssim B\).
Only derivatives of smooth finite-rank vectors are used there.
The cubic Taylor remainder is bounded by
\(C_q(N_3+N_3^2)N_q\), using the actual bound \(N_3\lesssim B/D\).
The phase term is \(O(B^2N_q)\).  Finally
\eqref{uniq:modulation-homogeneous} gives
\(\|f-f^0\|_{q,\lambda}\lesssim_q BN_1+N_1^2\).
Together these prove \eqref{uniq:g-homogeneous}; in particular no
forcing term independent of \(w\) is left in that estimate.

Define the quasistatic correction
\[
 r^0=-\mathcal A^{-1}g^0\in W.
\]
Lemma~\ref{uniq:inverse} gives
\(\|r^0\|_{q,\lambda}\le C_qB^2\) at every fixed order.
Its derivative must account for the moving constraint space.
Since \(r^0=\Pi r^0\),
\[
 (I-\Pi)r_s^0=\Pi_s r^0.
\]
Differentiate \(\mathcal Ar^0=-g^0\), apply \(\Pi\), and use
\(\mathcal A\Pi_s r^0=0\).  This gives the exact covariant formula
\begin{equation}\label{uniq:covariant-correction}
 \Pi r_s^0
 =-\mathcal A^{-1}\Pi(g_s^0+\mathcal A_s r^0).
\end{equation}
The old outgoing estimates imply
\[
 |D_s|\le CDB,\qquad |b_s|\le CDB^2,\qquad
 \|\mathcal A_s h\|_{q,\lambda}
       \le C_q\kappa\|h\|_{q+2,\lambda}.
\]
Thus \eqref{uniq:g0-bounds} and \eqref{uniq:covariant-correction}
give
\begin{equation}\label{uniq:correction-s}
 \|\Pi r_s^0\|_{q,\lambda}\le C_qD^{p_q}B^3.
\end{equation}
Only derivatives of the explicit profile and the already controlled
parameters occur; no small time derivative of \(w\) is being assumed.

Set \(z=y-r^0\).  Equations
\eqref{uniq:projected-single} and \eqref{uniq:covariant-correction}
yield
\[
 z_s=\mathcal Az+\Pi_s z+\widetilde g,\qquad
 \widetilde g=g-g^0-\Pi r_s^0\in W,
\]
with
\[
 \|\widetilde g\|_{q,\lambda}
 \le C_qB\|z\|_{q,\lambda}+C_qD^{p_q}B^3.
\]
For odd \(q\ge3\), use the constrained graph energy
\[
 X_q^2=\sum_{j=0}^{(q-1)/2}
           \langle\mathcal L\mathcal A^jz,\mathcal A^jz\rangle
       \asymp_q\|z\|_{q,\lambda}^2.
\]
The graph-energy identity proved in Section~\ref{sec:construction},
including its moving-projection connection, gives
\[
 |(X_q)_s|\le C_qB X_q+C_qD^{p_q}B^3.
\]
Since \(D_s\asymp DB>0\),
\begin{equation}\label{uniq:quasistatic-gronwall}
 |(X_q)_D|\le\frac{C_q}{D}X_q+C_qD^{p_q-1}B^2.
\end{equation}
At an arbitrary distant section \(D_2\), the existing infinite-orbit
estimate \eqref{uniq:class} gives
\[
 X_q(D_2)\le C_qB(D_2)/D_2.
\]
Backward integration of
\eqref{uniq:quasistatic-gronwall} therefore yields
\[
 X_q(D)\le (D_2/D)^{C_q}X_q(D_2)
   +C_q\int_D^{D_2}(r/D)^{C_q}r^{p_q-1}B(r)^2\,dr.
\]
The boundary term vanishes as \(D_2\to\infty\).  The asymptotic
formula for \(B\) gives
\[
 \frac{B(r)^2}{B(D)^2}
 \le C(r/D)^{-1/2}e^{-(r-D)}\qquad(r\ge D\gg1),
\]
so the integral is bounded by a finite polynomial in \(D\) times
\(B(D)^2\).  Adding \(r^0\), and then the quadratic normal component,
proves \eqref{uniq:two-exponential} at odd orders.  Every remaining
fixed order follows from the next odd order.

This argument uses the actual infinite-orbit boundary estimate.  At a
finite terminal section where \(w=0\), the corrected remainder is
\(z=-r^0\), and no zero boundary condition for \(z\) has been used.
\end{proof}

\subsection{Comparison at a common scale}

Write \(t_j(D)\) for the inverse of \(D_j(t)\) on its increasing
outgoing tail. Writing a parameter or remainder as a function of
\(D\) means composition with this inverse.

Let the two solutions in Theorem~\ref{thm:uniqueness} now be evaluated
at their own crossings of a common sufficiently large scale:
\[
 z_j(D)=e^{-i\gamma_j(t_j(D))}u_j(t_j(D))
       =V_{D,b_j}+w_j,
 \qquad d=w_1-w_2,
 \qquad v=z_1-z_2.
\]
At the common scale, write \(V_j=V_{D,b_j(D)}\),
\(g_{0,j}=V_j\), and
\(g_{\ell,j}=I_DT_\ell(D,b_j(D))\) for \(1\le\ell\le3\).
The spaces and operators \(W_j,S_j,\Pi_j,\mathcal L_j,\mathcal A_j\)
are those defined above, evaluated at \((D,b_j(D))\), with
\(\lambda=L/D\). Their \(D\)-derivatives are total derivatives
along this curve at fixed physical coordinate \(x\).
Put \(\delta b=b_1-b_2\) and use \(\delta\) for other differences.
Proposition~\ref{uniq:refined} and Lemmas
\ref{uniq:energy-map}--\ref{uniq:modulation-map} imply
\begin{equation}\label{uniq:slaving}
 |\delta b|+|\delta\mathbf m|
 \le CD^pB\|d\|_{1,\lambda}.
\end{equation}
Since \(\|V_{D,b_1}-V_{D,b_2}\|_{k,\lambda}\le C_k|\delta b|\),
absorption gives
\begin{equation}\label{uniq:d-v}
 \|d\|_{1,\lambda}\asymp\|v\|_{1,\lambda},\qquad
 |\delta b|+|\delta\mathbf m|
 \le CD^pB\|v\|_{1,\lambda}.
\end{equation}
In absolute size \(|\delta b|\le CD^pB^3\), and therefore
\begin{equation}\label{uniq:individual-v}
 \|v\|_{k,\lambda}\le C_kD^{p_k}B^2
\end{equation}
for every fixed order needed below.

Project the full-field difference using the first orbit:
\[
 \xi=\Pi_1v,\qquad n=(I-\Pi_1)v.
\]
Equal mass of \(z_1,z_2\) gives exactly
\begin{equation}\label{uniq:mass-difference}
 \langle v,V_1\rangle
       =\tfrac12\|v\|_2^2-\langle v,w_1\rangle.
\end{equation}
For the three other tests,
\begin{equation}\label{uniq:test-difference}
 \langle v,g_{\ell,1}\rangle
 =\langle V_1-V_2,g_{\ell,1}\rangle
  -\langle w_2,g_{\ell,1}-g_{\ell,2}\rangle
 \qquad(1\le\ell\le3).
\end{equation}
Finite-rank reconstruction, \eqref{uniq:d-v}, and
\eqref{uniq:individual-v} consequently give
\begin{equation}\label{uniq:normal-difference}
 \|n\|_{k,\lambda}\le C_kD^{p_k}B\|v\|_{1,\lambda},\qquad
 \|v\|_{1,\lambda}\asymp\|\xi\|_{1,\lambda}.
\end{equation}
The normal component here is the normal component of a difference of
full fields; it has not been replaced by the quadratic normal component
of a single remainder.

Two algebraic cancellations are needed before estimating its evolution.
First, the entire limiting normal space is invariant under the
Euclidean Hamiltonian generator.  In the ordered basis
\[
 -iQ,\quad-i|y|^2Q,\quad\rho,\quad\Lambda Q,
\]
the images under \(J\mathcal L_\infty\) are respectively
\begin{equation}\label{uniq:jordan-chain}
 0,\quad4\Lambda Q,\quad-i|y|^2Q/4,\quad2iQ.
\end{equation}
These follow from
\(L_-Q=0\), \(L_-(|y|^2Q)=-4\Lambda Q\),
\(L_+\rho=|y|^2Q/4\), and \(L_+\Lambda Q=-2Q\).
Apply these identities to each phase-consistent localized image.
The chirp changes their left sides by \(O_k(B)\), and the mass
correction and mixed image products by \(O_k(D^{p_k}B^2)\).
Uniform invertibility of the four-dimensional basis then gives
\begin{equation}\label{uniq:normal-leakage}
 \|\Pi_1J\mathcal L_1h\|_{k,\lambda}
 \le C_kB\|h\|_2\qquad(h\in\ker\Pi_1).
\end{equation}
This is a smoothing finite-rank estimate, so it requires no derivatives
of \(h\).

Second, define the physical scaled vector field
\[
 F_\lambda(z)=J(-\lambda^2\Delta z+z-\lambda^2\mathcal N(z)).
\]
The two leading terms in \eqref{uniq:static-remainder} are annihilated
by \(\Pi J\).  It follows that
\(\|\Pi_jF_\lambda(V_j)\|_{k,\lambda}\le C_kB^2\).
Moreover
\[
 \|\Pi_1-\Pi_2\|_{L^2\to H^k_\lambda}\le C_k|\delta b|.
\]
Using \(\|F_\lambda(V_2)\|_{k,\lambda}\le C_kB\) and the improved
known remainder bounds, we obtain
\begin{align}
 \|\Pi_1F_\lambda(z_2)\|_{1,\lambda}
 &\le C\bigl(B^2+B|\delta b|+\|w_2\|_{3,\lambda}
                  +\|w_2\|_{3,\lambda}^2+\|w_2\|_{3,\lambda}^3\bigr)
 \le CD^pB^2.\label{uniq:projected-field}
\end{align}
The two derivatives fall only on the known second remainder.
The corresponding phase cancellation is
\begin{equation}\label{uniq:phase-cancellation}
 \Pi_1JV_1=0,
 \qquad \|\Pi_1Jz_2\|_{1,\lambda}
       \le C(\|w_2\|_{1,\lambda}+|\delta b|)
       \le CD^pB^2.
\end{equation}

We now write the exact equation, retaining the two different clocks.
For each orbit set
\[
 \kappa_j=\frac1{D(b_j-\mu_j)}.
\]
Its phase-removed field satisfies
\begin{equation}\label{uniq:common-scale}
 (z_j)_D=\kappa_j\{F_\lambda(z_j)+\omega_jJz_j\},\qquad
 (t_j)_D=\lambda^2\kappa_j,
 \qquad(\gamma_j)_D=(1+\omega_j)\kappa_j.
\end{equation}
In particular
\begin{equation}\label{uniq:clock-difference}
 \delta\kappa
 =-\frac{\delta b-\delta\mu}
          {D(b_1-\mu_1)(b_2-\mu_2)},\qquad
 |\delta\kappa|\le\frac{CD^p}{DB}\|\xi\|_{1,\lambda}.
\end{equation}
The denominator in this identity is of order \(DB^2\).

Let
\[
 \mathcal T=\mathcal N(z_1)-\mathcal N(z_2)
                         -D\mathcal N(V_1)v.
\]
After applying \(I_D^{-1}\), we use the uniform embedding
\(H^3(\mathbb T_D^2)\hookrightarrow W^{1,\infty}(\mathbb T_D^2)\).
The dimensionless fields \(I_D^{-1}(z_j-V_1)\) have size
\(D^pB^2\) in that space, by \eqref{uniq:two-exponential} and
\eqref{uniq:slaving}.  Expanding the cubic difference in these
variables, and using the identity
\(I_D^{-1}(\lambda^2\mathcal N(z))=\mathcal N(I_D^{-1}z)\), gives
\begin{equation}\label{uniq:cubic-difference}
 \|\lambda^2\mathcal T\|_{1,\lambda}
       \le CD^pB^2\|v\|_{1,\lambda}.
\end{equation}
Only \(H^1\) of the unknown difference is used in this inequality.

Subtract \eqref{uniq:common-scale} and differentiate \(\Pi_1v\).
With \(\mathcal A_1=\Pi_1J\mathcal L_1\Pi_1\), the result is
\begin{equation}\label{uniq:projected-difference}
 \xi_D=\kappa_1\mathcal A_1\xi+(\Pi_1)_D\xi+G,
\end{equation}
where
\begin{align}
 G={}&(\Pi_1)_Dn+\kappa_1\Pi_1J\mathcal L_1n
      -\kappa_1\lambda^2\Pi_1J\mathcal T
      +\kappa_1\omega_1\Pi_1Jv\notag\\
    &+\delta\kappa\,\Pi_1\{F_\lambda(z_2)+\omega_2Jz_2\}
      +\kappa_1\delta\omega\,\Pi_1Jz_2.
 \label{uniq:all-difference-terms}
\end{align}
Every term of \(G\) lies in \(W_1\).  For the first term this follows
from the derivative of \(\Pi_1^2=\Pi_1\), since \(\Pi_1n=0\).
The separate connection \((\Pi_1)_D\xi\) in
\eqref{uniq:projected-difference} has not been dropped.

Along an individual orbit,
\[
 (\log\lambda)_D=-D^{-1},\qquad
 b_D=\frac{a-b^2+2b\mu+\nu}{D(b-\mu)}=O(B).
\]
Parameter differentiation of the explicit physical profiles therefore
gives
\begin{equation}\label{uniq:operator-D}
 \|(\Pi_1)_Dh\|_{1,\lambda}\le\frac CD\|h\|_2,
 \qquad
 |\langle(\mathcal L_1)_Dh,h\rangle|
       \le\frac CD\|h\|_{1,\lambda}^2.
\end{equation}
All terms in \eqref{uniq:all-difference-terms} can now be estimated:
\begin{equation}\label{uniq:G-bound}
 \|G\|_{1,\lambda}
       \le\frac{CD^pB}{D}\|\xi\|_{1,\lambda}.
\end{equation}
For completeness, the first term uses
\eqref{uniq:normal-difference} and \eqref{uniq:operator-D}.
The second uses the extra \(B\) in
\eqref{uniq:normal-leakage}, together with
\(|\kappa_1|\lesssim(DB)^{-1}\).
The cubic term follows from \eqref{uniq:cubic-difference}, and the
term containing \(\omega_1v\) uses \(|\omega_1|\lesssim B^2\).
The clock term is controlled by the product
\[
 |\delta\kappa|\,\|\Pi_1F_\lambda(z_2)\|_{1,\lambda}
       \le\frac{CD^pB}{D}\|\xi\|_{1,\lambda},
\]
using \eqref{uniq:projected-field} before estimating the clock
difference.  Its remaining phase factor and the last term in
\eqref{uniq:all-difference-terms} are controlled by
\eqref{uniq:phase-cancellation} and \eqref{uniq:d-v}.
Thus every polynomial prefactor left by the refined individual bounds
still has a factor \(B\).

\subsection{The homogeneous estimate and the normalized terminal limit}

\begin{proof}[Proof of Theorem~\ref{thm:uniqueness}]
On the common large-scale tail define the positive difference energy
\[
 \mathfrak D(D)=\langle\mathcal L_1\xi,\xi\rangle
               \asymp\|\xi\|_{1,\lambda}^2.
\]
Since \((\Pi_1J)^*=-\Pi_1J\),
\[
 \langle\mathcal L_1\xi,\mathcal A_1\xi\rangle=0.
\]
Differentiating the energy with
\eqref{uniq:projected-difference} therefore gives exactly
\[
 \mathfrak D_D
 =\langle(\mathcal L_1)_D\xi,\xi\rangle
 +2\langle\mathcal L_1\xi,(\Pi_1)_D\xi+G\rangle.
\]
The scalar clock coefficient \(\kappa_1\) has canceled without being
differentiated.  The pairings are the continuous \(H^1\) quadratic
form pairings.  Smoothness of the solutions justifies the displayed
identities directly; the estimates depend only on the stated low-order
norms of the difference.

Form continuity, \eqref{uniq:operator-D}, and \eqref{uniq:G-bound}
imply, after increasing the starting scale if necessary,
\begin{equation}\label{uniq:homogeneous-energy}
 |\mathfrak D_D|\le\frac{C_*}{D}\mathfrak D.
\end{equation}
Hence for \(D_2\ge D_1\) on that tail,
\[
 \mathfrak D(D_1)\le (D_2/D_1)^{C_*}\mathfrak D(D_2).
\]
There is no division by a possibly vanishing difference norm in this
argument.  By \eqref{uniq:individual-v},
\(\mathfrak D(D_2)\le C D_2^pB(D_2)^4\).
Letting \(D_2\to\infty\) proves \(\mathfrak D=0\), and
\eqref{uniq:normal-difference} gives \(\xi=v=0\).
Equation \eqref{uniq:d-v} then yields \(\delta b=0\) and
\(d=0\); the exact modulation map gives
\(\delta\mu=\delta\nu=\delta\omega=0\).

The scalar identities in \eqref{uniq:common-scale} now imply
\[
 (t_2-t_1)_D=0,\qquad(\gamma_1-\gamma_2)_D=0.
\]
Call these constants \(\tau\) and \(\theta\), respectively.
Since \(z_1=z_2\),
\begin{equation}\label{uniq:phase-time}
 u_1(t_1(D))=e^{i\theta}u_2(t_1(D)+\tau).
\end{equation}
Ordinary uniqueness of the smooth Cauchy problem extends this identity
wherever both sides exist, proving Theorem~\ref{thm:uniqueness}.
\end{proof}

Uniqueness also removes the subsequence used in the construction.
The time and phase normalizations at the initial scale are essential
for this conclusion.

\begin{corollary}[Convergence of the normalized terminal family]
\label{uniq:full-sequence}
Normalize the finite-terminal solutions of
Section~\ref{sec:construction} by assigning physical time zero to their
common outgoing scale section \(D_0\), and by making their modulation
phase zero at that section.  The entire resulting sequence of initial
data converges in every fixed \(H^k\) to one smooth datum.  No
subsequence selection remains in the normalized construction.
\end{corollary}

\begin{proof}
If the original member has phase \(\gamma_n(D_0)\), the phase
normalization means replacing it by
\(e^{-i\gamma_n(D_0)}u_n\).  This exact symmetry preserves its
equation, invariants, and all estimates.
The uniform bounds at every order proved in
Section~\ref{sec:construction} make the normalized initial data
precompact in each fixed Sobolev space.  Every subsequential limit has
the same centre, reflections, mass, energy, and outgoing bounds
\eqref{uniq:class}.  Theorem~\ref{thm:uniqueness} identifies any two
such limits up to a time shift and constant phase.

Both constants vanish under the stated normalization.  The scale
is strictly increasing on the common constructed tube.  If the time
shift were nonzero, exchange the two solutions if necessary so that
it is positive.  Cauchy uniqueness extends the identity back to time
zero, where the first representation has scale \(D_0\) and the second
has scale \(D_2>D_0\).

To apply local coordinate uniqueness to these two representations,
first compare their gradient norms.  Uniformly throughout the
constructed tube,
\[
 \|\nabla u\|_2
 =\frac{D}{L}\bigl(\|\nabla Q\|_2+o_{D_0\to\infty}(1)\bigr).
\]
The profile estimates and the uniform \(B/D\) remainder bound give
this formula.  Equality of the two fields implies
\(D_2/D_0=1+o_{D_0\to\infty}(1)\).  Their \(L^2\) profile overlap
also makes their phases close modulo \(2\pi\).  Thus, for the fixed
large starting section used in the construction, both representations
belong to the same uniform implicit-function neighborhood.  Local
coordinate uniqueness forces \(D_2=D_0\), a contradiction.  The time
shift is zero.  The common phase value zero at that section then fixes
the constant phase modulo \(2\pi\), which is the same phase factor.

Finally, if convergence of the full sequence failed in a fixed \(H^k\),
a subsequence would stay a positive distance from the unique limiting
datum in that norm.  Its uniform \(H^{k+1}\) bound would yield a
convergent subsubsequence, whose limit must be that same datum.  This
contradiction proves the assertion at every fixed order.
\end{proof}

The same difference estimate gives a rate of convergence for the terminal
construction.  Here it is important that each terminal datum has exactly
the limiting mass and energy.  We use \(N\) for its prescribed terminal
scale, so that \(N\) is a real parameter rather than a Fourier cutoff.

\begin{corollary}[Quantitative terminal approximation]\label{rate:corollary}
Choose the common section \(D_0\) sufficiently large, once and for all.
For \(N>D_0\), let \(u_N\) be the solution whose terminal profile at
scale \(N\) is \(V_{N,B(N)}\), normalized by
\[
 t_N(D_0)=0,\qquad \gamma_N(D_0)=0.
\]
Let \(u\) be the limiting solution with the same normalization.
There are constants \(C,P\), independent of \(N\), such that
\begin{equation}\label{rate:H1}
 \|u_N(0)-u(0)\|_{H^1(\mathbb T_L^2)}
                 \le C N^P e^{-N}.
\end{equation}
For every fixed integer \(k\ge0\), every \(0<\eta<1\), and every
compact interval \(I\subset\mathbb R\), there is a constant
\(C_{I,k,\eta}\) such that
\begin{equation}\label{rate:compact}
 \|u_N-u\|_{C(I;H^k(\mathbb T_L^2))}
                 \le C_{I,k,\eta}e^{-\eta N}
\end{equation}
for all sufficiently large \(N\).  The constants in these estimates
are independent of the terminal scale.
\end{corollary}

\begin{proof}
Propositions~\ref{con:terminal-family} and \ref{con:all-orders}
apply to every sufficiently large terminal scale, with the same
constants at each fixed order.  The normalization is by exact
symmetries and does not change those estimates.  Each \(u_N\)
also extends globally by Proposition~\ref{glob:continuation},
since its mass is \(m\) and its energy is zero.

For the finite solution, a subscript \(N\) on an operator
means evaluation at
\((D,b_N(D))\). In particular, \(g_N^0(D)=g^0(D,b_N(D))\),
and \(\mathcal A_N=\Pi_NJ\mathcal L_N\Pi_N\).
The current scale is \(\lambda=L/D\); at the terminal section
we write \(\lambda_N=L/N\).

We first need a finite-interval version of
Proposition~\ref{uniq:refined}.  On \(D_0\le D\le N\), put
\[
 y_N=\Pi_N w_N,\qquad
 r_N^0=-\mathcal A_N^{-1}g_N^0,\qquad
 z_N^{\mathrm{err}}=y_N-r_N^0.
\]
The graph energy of this corrected remainder is
\[
 X_q(D)^2=\sum_{a=0}^{(q-1)/2}
 \big\langle\mathcal L_N\mathcal A_N^a z_N^{\mathrm{err}},
                  \mathcal A_N^a z_N^{\mathrm{err}}\big\rangle,
 \qquad q\ge3\ \hbox{odd}.
\]
Its dependence on \(N\) is suppressed in the notation.
All identities used to derive
\eqref{uniq:quasistatic-gronwall} are local identities.
The old bounds \eqref{con:uniform-high}, with their constants
independent of \(N\), therefore give for the graph-energy square
root \(X_q\) of \(z_N^{\mathrm{err}}\), at any fixed odd order
\(q\ge3\),
\begin{equation}\label{rate:finite-quasistatic}
 |(X_q)_D|\le \frac{C_q}{D}X_q+
                         C_qD^{p_q}B(D)^2.
\end{equation}
The inverse and coefficient estimates are uniform on this interval.
There is a nonzero terminal boundary term:
\[
 w_N(N)=0,\qquad
 z_N^{\mathrm{err}}(N)=-r_N^0(N),\qquad
 X_q(N)\le C_qN^{p_q}B(N)^2.
\]
In particular, the corrected error is not assigned a zero terminal
value.  Integrating \eqref{rate:finite-quasistatic} backwards gives
\begin{equation}\label{rate:finite-tail}
 \begin{split}
 X_q(D)\le{}&
 C_q(N/D)^{C_q}N^{p_q}B(N)^2\\
 &+C_q\int_D^N(r/D)^{C_q}r^{p_q}B(r)^2\,dr.
 \end{split}
\end{equation}
Both terms are bounded by \(C_qD^{P_q}B(D)^2\), uniformly in
\(N\ge D\).  Indeed,
\[
 \frac{B(r)^2}{B(D)^2}
 \le C(r/D)^{-1/2}e^{-(r-D)}\qquad(r\ge D\ge D_0).
\]
Writing \(r=D+x\), the remaining factors are bounded by a fixed
power of \(D\) times \((1+x/D)^a e^{-x}\).  The latter function
has a uniform supremum and integral for each fixed \(a\).
This estimates both the endpoint and the integral in
\eqref{rate:finite-tail}.
Adding \(r_N^0\) and then using the quadratic normal estimate
\eqref{mod:normal-part}, we obtain
\begin{equation}\label{rate:finite-refined}
 \|w_N(D)\|_{q,\lambda}
       \le C_qD^{P_q}B(D)^2,\qquad D_0\le D\le N,
\end{equation}
with constants independent of \(N\).
Even and lower orders follow from the next odd order.

We now compare \(u_N\) and \(u\) at the same scale, with their
own physical clocks and modulation phases.  Denote their
phase-removed full fields by
\[
 Z_N(D)=V_{D,b_N}+w_N,\qquad
 Z(D)=V_{D,b}+w,
\]
and let
\[
 \xi_N=\Pi_N(Z_N-Z),\qquad
 \mathfrak D_N(D)=\langle\mathcal L_N\xi_N,\xi_N\rangle.
\]
The proof of \eqref{uniq:homogeneous-energy} uses the local
mass and energy constraints, the low-order modulation estimates,
and the improved remainder bounds.  It uses the infinite interval
only in its final passage of the boundary term to zero.
Here \eqref{rate:finite-refined} supplies the improved bounds
for the finite solution, with constants independent of \(N\).
Thus every term in \eqref{uniq:all-difference-terms}, including
the different clocks and phases, obeys the same estimate on
\([D_0,N]\).  Enlarging the fixed \(D_0\) for this finite
collection of low-order estimates, we have
\begin{equation}\label{rate:finite-difference}
 \mathfrak D_N\asymp\|Z_N-Z\|_{1,\lambda}^2,\qquad
 |(\mathfrak D_N)_D|\le\frac{C_*}{D}\mathfrak D_N,
\end{equation}
where \(C_*\) is independent of \(N\).
No assertion that \(u_N\) belongs to the infinite outgoing
class is needed.

At \(D=N\), the finite solution has \(b_N=B(N)\) and \(w_N=0\).
The infinite solution has
\(\|w(N)\|_{1,\lambda_N}\le CN^pB(N)^2\).
Applying Lemma~\ref{uniq:energy-map} to these two exactly
zero-energy fields gives
\[
 |B(N)-b(N)|\le CN^pB(N)^3.
\]
The smooth dependence of \(V_{N,b}\) on \(b\) consequently gives
\[
 \mathfrak D_N(N)^{1/2}\le CN^pB(N)^2.
\]
Integrating \eqref{rate:finite-difference} without discarding
this terminal error yields
\[
 \mathfrak D_N(D_0)^{1/2}
 \le (N/D_0)^{C_*/2}\mathfrak D_N(N)^{1/2}
 \le C N^P e^{-N}.
\]
At \(D_0\), both physical times and both modulation phases are
zero.  Moreover, \(\lambda_0=L/D_0\) is fixed.  Therefore
\(Z_N(D_0)-Z(D_0)=u_N(0)-u(0)\), and the scaled norm is
equivalent to the physical \(H^1\) norm.  This proves
\eqref{rate:H1}.

For every fixed integer \(K\), the common-section estimates give
\[
 \sup_{N\text{ large}}\|u_N(0)\|_{H^K}+\|u(0)\|_{H^K}<\infty.
\]
If \(1<k<K\), interpolation between \(H^1\) and \(H^K\) gives
\[
 \|u_N(0)-u(0)\|_{H^k}
 \le C_K\bigl(N^Pe^{-N}\bigr)^{(K-k)/(K-1)}.
\]
Given \(\eta<1\), choose \(K\) so that
\((K-k)/(K-1)>\eta\); the polynomial factor is then absorbed
by the remaining exponential margin.
The cases \(k=0,1\) follow directly from \eqref{rate:H1}.
Hence the initial-data estimate in every fixed order is
\(C_{k,\eta}e^{-\eta N}\).

Finally, choose an integer \(j\ge\max\{k,2\}\) and a compact
interval \(J\) containing both \(I\) and zero.  The limiting
solution is bounded in \(H^j\) on \(J\).  The algebra property
of \(H^j(\mathbb T_L^2)\) and Duhamel's formula give, on a
bounded \(H^j\) neighborhood of this solution,
\[
 \|u_N(t)-u(t)\|_{H^j}
 \le \|u_N(0)-u(0)\|_{H^j}
 +C_J\left|\int_0^t\|u_N(r)-u(r)\|_{H^j}\,dr\right|.
\]
A bootstrap and Gronwall keep \(u_N\) in that neighborhood on
\(J\), for all sufficiently large \(N\), and preserve the
initial exponential rate.  The argument is applied in both
time directions from zero.  Restricting to \(I\) and using the
embedding into \(H^k\) proves \eqref{rate:compact}.
\end{proof}


\begin{thebibliography}{99}

\bibitem{Banica04}
V.~Banica,
\emph{Remarks on the blow-up for the Schr\"odinger equation with critical mass on a plane domain},
Ann. Sc. Norm. Super. Pisa Cl. Sci. (5) \textbf{3} (2004), no.~1, 139--170.

\bibitem{BCD11}
V.~Banica, R.~Carles and T.~Duyckaerts,
\emph{Minimal blow-up solutions to the mass-critical inhomogeneous NLS equation},
Comm. Partial Differential Equations \textbf{36} (2011), no.~3, 487--531.

\bibitem{Bourgain93}
J.~Bourgain,
\emph{Fourier transform restriction phenomena for certain lattice subsets and applications to nonlinear evolution equations. I. Schr\"odinger equations},
Geom. Funct. Anal. \textbf{3} (1993), no.~2, 107--156.

\bibitem{Bourgain96}
J.~Bourgain,
\emph{On the growth in time of higher Sobolev norms of smooth solutions of Hamiltonian PDE},
Internat. Math. Res. Notices \textbf{1996}, no.~6, 277--304.

\bibitem{Bourgain00}
J.~Bourgain,
\emph{Problems in Hamiltonian PDE's},
Geom. Funct. Anal., Special Volume, Part~I (2000), 32--56.

\bibitem{BourgainWang97}
J.~Bourgain and W.~Wang,
\emph{Construction of blowup solutions for the nonlinear Schr\"odinger equation with critical nonlinearity},
Ann. Scuola Norm. Sup. Pisa Cl. Sci. (4) \textbf{25} (1997), no.~1--2, 197--215.

\bibitem{BrezisGallouet80}
H.~Br\'ezis and T.~Gallou\"et,
\emph{Nonlinear Schr\"odinger evolution equations},
Nonlinear Anal. \textbf{4} (1980), no.~4, 677--681.

\bibitem{BGT03}
N.~Burq, P.~G\'erard and N.~Tzvetkov,
\emph{Two singular dynamics of the nonlinear Schr\"odinger equation on a plane domain},
Geom. Funct. Anal. \textbf{13} (2003), no.~1, 1--19.

\bibitem{BGT04}
N.~Burq, P.~G\'erard and N.~Tzvetkov,
\emph{Strichartz inequalities and the nonlinear Schr\"odinger equation on compact manifolds},
Amer. J. Math. \textbf{126} (2004), no.~3, 569--605.

\bibitem{Cazenave03}
T.~Cazenave,
\emph{Semilinear Schr\"odinger Equations},
Courant Lecture Notes in Mathematics, vol.~10,
New York University, Courant Institute of Mathematical Sciences, New York;
American Mathematical Society, Providence, RI, 2003.

\bibitem{ChenPlane}
X. Chen,
\emph{Existence of modified wave operators and infinite cascade result for
a half wave Schr\"odinger equation on the plane},
J. Funct. Anal. \textbf{286} (2024), no.~2, Paper No.~110222, 74~pp.
\bibitem{Chen26.1}
X. Chen,
\emph{Infinite cascades for the cubic nonlinear half-wave equation on the line}, preprint, \newblock hal-05731979, 2026.

\bibitem{Chen26.2}
X. Chen, \emph{Infinite cascades for the defocusing cubic half-wave equation on the torus}, preprint, \newblock arXiv:2609.24930, 2026. 

\bibitem{CKSTT10}
J.~Colliander, M.~Keel, G.~Staffilani, H.~Takaoka and T.~Tao,
\emph{Transfer of energy to high frequencies in the cubic defocusing nonlinear Schr\"odinger equation},
Invent. Math. \textbf{181} (2010), no.~1, 39--113.

\bibitem{CMM11}
R.~C\^ote, Y.~Martel and F.~Merle,
\emph{Construction of multi-soliton solutions for the \(L^2\)-supercritical gKdV and NLS equations},
Rev. Mat. Iberoam. \textbf{27} (2011), no.~1, 273--302.

\bibitem{Dodson15}
B.~Dodson,
\emph{Global well-posedness and scattering for the mass critical nonlinear Schr\"odinger equation with mass below the mass of the ground state},
Adv. Math. \textbf{285} (2015), 1589--1618.

\bibitem{GG10}
P.~G\'erard and S.~Grellier,
\emph{The cubic Szeg\H{o} equation},
Ann. Sci. \'Ec. Norm. Sup\'er. (4) \textbf{43} (2010), no.~5, 761--810.

\bibitem{GG12}
P.~G\'erard and S.~Grellier,
\emph{Effective integrable dynamics for a certain nonlinear wave equation},
Anal. PDE \textbf{5} (2012), no.~5, 1139--1155.

\bibitem{GG17}
P.~G\'erard and S.~Grellier,
\emph{The Cubic Szeg\H{o} Equation and Hankel Operators},
Ast\'erisque \textbf{389}, Soci\'et\'e Math\'ematique de France, Paris, 2017.

\bibitem{GLPR18}
P.~G\'erard, E.~Lenzmann, O.~Pocovnicu and P.~Rapha\"el,
\emph{A two-soliton with transient turbulent regime for the cubic half-wave equation on the real line},
Ann. PDE \textbf{4} (2018), no.~1, Paper No.~7.

\bibitem{GK15}
M.~Guardia and V.~Kaloshin,
\emph{Growth of Sobolev norms in the cubic defocusing nonlinear Schr\"odinger equation},
J. Eur. Math. Soc. \textbf{17} (2015), no.~1, 71--149;
erratum, \textbf{19} (2017), no.~2, 601--602.

\bibitem{Hani14}
Z.~Hani,
\emph{Long-time instability and unbounded Sobolev orbits for some periodic nonlinear Schr\"odinger equations},
Arch. Ration. Mech. Anal. \textbf{211} (2014), no.~3, 929--964.

\bibitem{HPTV15}
Z.~Hani, B.~Pausader, N.~Tzvetkov and N.~Visciglia,
\emph{Modified scattering for the cubic Schr\"odinger equation on product spaces and applications},
Forum Math. Pi \textbf{3} (2015), e4, 63~pp.

\bibitem{HK2026}
S.~Herr, B.~Kwak, \emph{Global well-posedness of the cubic nonlinear Schrödinger equation on $\mathbb{T}^2$}, Inventiones mathematicae (2026): 1-117.
\bibitem{KLR13}
J.~Krieger, E.~Lenzmann and P.~Rapha\"el,
\emph{Nondispersive solutions to the \(L^2\)-critical half-wave equation},
Arch. Ration. Mech. Anal. \textbf{209} (2013), no.~1, 61--129.

\bibitem{KMR09}
J.~Krieger, Y.~Martel and P.~Rapha\"el,
\emph{Two-soliton solutions to the three-dimensional gravitational Hartree equation},
Comm. Pure Appl. Math. \textbf{62} (2009), no.~11, 1501--1550.

\bibitem{Kwong89}
M.~K.~Kwong,
\emph{Uniqueness of positive solutions of \(\Delta u-u+u^p=0\) in \(\mathbb R^n\)},
Arch. Ration. Mech. Anal. \textbf{105} (1989), no.~3, 243--266.

\bibitem{LeCozMR16}
S.~Le~Coz, Y.~Martel and P.~Rapha\"el,
\emph{Minimal mass blow up solutions for a double power nonlinear Schr\"odinger equation},
Rev. Mat. Iberoam. \textbf{32} (2016), no.~3, 795--833.

\bibitem{MartelMerle06}
Y.~Martel and F.~Merle,
\emph{Multi solitary waves for nonlinear Schr\"odinger equations},
Ann. Inst. H. Poincar\'e C Anal. Non Lin\'eaire \textbf{23} (2006), no.~6, 849--864.

\bibitem{MMR15}
Y.~Martel, F.~Merle and P.~Rapha\"el,
\emph{Blow up for the critical gKdV equation. III: Exotic regimes},
Ann. Sc. Norm. Super. Pisa Cl. Sci. (5) \textbf{14} (2015), no.~2, 575--631.

\bibitem{MartelRaphael18}
Y.~Martel and P.~Rapha\"el,
\emph{Strongly interacting blow up bubbles for the mass critical nonlinear Schr\"odinger equation},
Ann. Sci. \'Ec. Norm. Sup\'er. (4) \textbf{51} (2018), no.~3, 701--737.

\bibitem{Merle90}
F.~Merle,
\emph{Construction of solutions with exactly \(k\) blow-up points for the Schr\"odinger equation with critical nonlinearity},
Comm. Math. Phys. \textbf{129} (1990), no.~2, 223--240.

\bibitem{Merle93}
F.~Merle,
\emph{Determination of blow-up solutions with minimal mass for nonlinear Schr\"odinger equations with critical power},
Duke Math. J. \textbf{69} (1993), no.~2, 427--454.

\bibitem{MR05}
F.~Merle and P.~Rapha\"el,
\emph{The blow-up dynamic and upper bound on the blow-up rate for critical nonlinear Schr\"odinger equation},
Ann. of Math. (2) \textbf{161} (2005), no.~1, 157--222.

\bibitem{MR06}
F.~Merle and P.~Rapha\"el,
\emph{On a sharp lower bound on the blow-up rate for the \(L^2\) critical nonlinear Schr\"odinger equation},
J. Amer. Math. Soc. \textbf{19} (2006), no.~1, 37--90.

\bibitem{Pocovnicu11}
O.~Pocovnicu,
\emph{Explicit formula for the solution of the Szeg\H{o} equation on the real line and applications},
Discrete Contin. Dyn. Syst. \textbf{31} (2011), no.~3, 607--649.

\bibitem{RaphaelSzeftel11}
P.~Rapha\"el and J.~Szeftel,
\emph{Existence and uniqueness of minimal blow-up solutions to an inhomogeneous mass critical NLS},
J. Amer. Math. Soc. \textbf{24} (2011), no.~2, 471--546.

\bibitem{Weinstein83}
M.~I.~Weinstein,
\emph{Nonlinear Schr\"odinger equations and sharp interpolation estimates},
Comm. Math. Phys. \textbf{87} (1983), no.~4, 567--576.

\bibitem{Weinstein85}
M.~I.~Weinstein,
\emph{Modulational stability of ground states of nonlinear Schr\"odinger equations},
SIAM J. Math. Anal. \textbf{16} (1985), no.~3, 472--491.

\bibitem{XuWaveguide}
H. Xu,
\emph{Unbounded Sobolev trajectories and modified scattering theory for
a wave guide nonlinear Schr\"odinger equation},
Math. Z. \textbf{286} (2017), nos.~1--2, 443--489.
\end{thebibliography}
\end{document}